\documentclass[a4paper,11pt,titlepage,twoside]{article}

\usepackage{graphicx}
\usepackage[T1]{fontenc}
\usepackage[utf8]{inputenc}
\usepackage[english]{babel}
\usepackage{soul}
\usepackage{amsfonts}
\usepackage{amsmath}
\usepackage{amsthm}
\usepackage{amssymb}
\usepackage{mathrsfs}
\usepackage[top=5cm, bottom=2.5cm, left=4cm, right=4cm]{geometry}
\usepackage{setspace}
\usepackage{afterpage}
\usepackage{extarrows}
\usepackage{fancyhdr}
\usepackage{titlesec}

\theoremstyle{definition}
\newtheorem{defin}{Definition}[section]
\theoremstyle{plain}
\newtheorem{teor}[defin]{Theorem}
\newtheorem{lem}[defin]{Lemma}
\newtheorem{pro}[defin]{Proposition}
\newtheorem{cor}[defin]{Corollary}
\theoremstyle{definition}
\newtheorem{esm}[defin]{Example}
\newtheorem{osr}[defin]{Remark}

\renewcommand{\O}{\Omega}
\newcommand{\D}{\mathcal{D}}
\renewcommand{\H}{\mathcal{H}}
\newcommand{\M}{\mathcal{M}}

\newcommand{\B}{\mathcal{B}}
\newcommand{\Ee}{\mathcal{E}}
\newcommand{\Eo}{\mathcal{E}_\Omega}
\newcommand{\rn}{\mathfrak{n}}

\newcommand{\T}{\Theta}
\newcommand{\Up}{\Upsilon}
\renewcommand{\L}{\Lambda}
\newcommand{\Po}{\mathfrak{P}}
\newcommand{\n}[1]{||#1||}
\newcommand{\nor}{||\cdot||}

\newcommand{\noo}{||\cdot||_\O}
\newcommand{\conv}[1]{\xrightarrow{#1}}

\renewcommand{\l}{\langle}
\renewcommand{\r}{\rangle}
\newcommand{\N}{\mathbb{N}}
\newcommand{\R}{\mathbb{R}}
\newcommand{\C}{\mathbb{C}}

\newcommand{\pint}{\l\cdot|\cdot\r}
\newcommand{\pin}[2]{\l#1 | #2\r}
\newcommand{\no}{\noindent}
\newcommand{\ol}{\overline}
\newcommand{\ull}{\underline}

\newcommand{\Ma}{\mathscr{M}_{\ull{\alpha}}}
\newcommand{\Oa}{\O_{\ull{\alpha}}}

\newcommand{\sub}{\subseteq}

\newcommand{\mez}{\frac{1}{2}}
\renewcommand{\ll}{{\it l}}

\titleformat{\section}
{\normalfont\fillast \fontsize{12}{15}\scshape}{\thesection.}{0.8em}{}

\begin{document}
	
\thispagestyle{plain}

\vspace*{1cm}

\begin{center}
	\large
	{\bf REPRESENTATION THEOREMS \\ FOR SOLVABLE SESQUILINEAR FORMS} \\
	\vspace*{0.5cm}
	ROSARIO CORSO AND CAMILLO TRAPANI
\end{center}

\normalsize
\vspace*{1cm}

\small

\begin{minipage}{11.8cm}
	{\scshape Abstract.} 	New results are added to the paper \cite{Tp_DB} about q-closed and solvable sesquilinear forms. The structure of the Banach space $\D[\noo]$ defined on the domain $\D$ of a q-closed sesquilinear form $\O$ is unique up to isomorphism, and the adjoint of a sesquilinear form has the same property of q-closure or of solvability. The operator associated to a solvable sesquilinear form is the greatest which represents the form and it is self-adjoint if, and only if, the form is symmetric.\\
	We give more criteria of solvability for q-closed sesquilinear forms. Some of these criteria are related to the numerical range, and we analyse in particular the  forms which are solvable with respect to inner products.\\
	The theory of solvable sesquilinear forms generalises those of many known sesquilinear forms in literature.
\end{minipage}

\vspace*{.5cm}

\begin{minipage}{11.8cm}
	{\scshape Keywords:} Kato's First Representation Theorem, q-closed/solvable sesquilinear forms, compatible norms, Banach-Gelfand triplet.
\end{minipage}

\vspace*{.5cm}

\begin{minipage}{11.8cm}
	{\scshape MSC (2010):} 47A07, 47A30.
	
\end{minipage}

\normalsize

\vspace*{.4cm}

\section{Introduction}
	
If $\O$ is a bounded sesquilinear form defined on a Hilbert space $\H$, with inner product $\pint$ and norm $\nor$, then, as it is well known, then there exists a unique bounded operator $T$ on $\H$ such that
\begin{equation}
\label{rapp_O_lim}
\O(\xi,\eta)=\pin{T\xi}{\eta} \qquad \forall \xi,\eta \in \H.
\end{equation}

\noindent In the unbounded case we are interested in determining a representation of type (\ref{rapp_O_lim}), but we confine ourselves to consider the forms defined on a dense subspace $\D$ of $\H$ and to search a closed operator $T$, with dense domain $D(T)\sub \D$ in $\H$ such that
\begin{equation*}
\label{rapp_O}
\O(\xi,\eta)=\pin{T\xi}{\eta} \qquad \forall \xi\in D(T),\eta \in \D.
\end{equation*}

One of the first results on representation of unbounded sesquilinear forms is Kato's Theorem \cite{Kato} for densely defined closed sectorial forms. The sectoriality is a condition on the numerical range of the form, while the closure indicates that in the domain is defined a Hilbert space, which is continuously embedded in $\H$, and the sesquilinear form is bounded in it.

The hypothesis that the domain $\D$ of a sesquilinear form $\O$ can be turned into a Hilbert space $\D[\noo]$, with the previous property, occurs in other representation theorems, as for instance, in McIntosh's papers \cite{McIntosh68,McIntosh70',McIntosh70}. Precisely, in \cite{McIntosh70',McIntosh70} McIntosh studied sesquilinear forms, called {\it closed}, which, up to a perturbation with scalar multiple of the inner product, are represented in $\D[\noo]$ with bijective operators. Instead, in \cite{McIntosh68} he assumed that the numerical range is contained in the half-plane $\{\lambda\in \C:\Re\lambda \geq 0\}$.

Also Fleige, Hassi and de Snoo assume, in \cite{FHdeS}, indirectly such hypothesis, since they consider a symmetric sesquilinear form which, up to a sum with a real multiple of the inner product, determines a Krein space on the domain of the form, and the topological structure of a Krein space is, by definition, the one of a Hilbert space.
Results that derive from the mentioned hypothesis can be found also in \cite{Schm}.

In \cite{GKMV,Schmitz} representation theorems are formulated for sesquilinear forms of the type
$$
\O(\xi,\eta)=\pin{HA^\mez\xi}{A^\mez\eta} \qquad \forall \xi,\eta \in D(A^\mez),
$$
where $H,A$ are self-adjoint operators satisfying some additional properties. In particular, $H$ is bounded with bounded inverse and $A$ is non-negative.

A different setting is followed by Arendt and ter Elst \cite{Arendt} in the framework of the $j$-{\it elliptic} forms. More precisely, assume that $(V,\nor_V)$ is a Hilbert space, $j:V\to \H$ a bounded operator with dense range in $\H$. $\O$ is a $j$-elliptic sesquilinear form if it is defined and bounded on $V$ and such that
$$
\Re \O(\xi,\xi)+\omega \n{j(\xi)}^2\geq \mu \n{\xi}_V^2 \qquad \forall \xi \in V,
$$
for some $\omega \in \R, \mu> 0$ ($\Re\O$ indicates the real part form of $\O$). Theorem 2.1 of \cite{Arendt} provides a representation
$$
\O(\xi,\eta)=\pin{Tj(\xi)}{j(\eta)} \qquad \forall \xi\in \D',\eta \in \D,
$$
where $T$ is a m-sectorial operator and $\D'$ is a subspace of $\D$.

Di Bella and Trapani in \cite{Tp_DB}, instead, have studied the sesquilinear form $\O$, called {\it q-closed}, on which domain $\D$ is defined a reflexive Banach space $\D[\noo]$, continuously embedded in $\H$, and such that the form is bounded in it. Under this condition, a {\it Banach-Gelfand triplet} $\D \hookrightarrow \H \hookrightarrow \D^\times$ is determined, where $\D^\times$ is the conjugate dual space of $\D[\noo]$.\\
Their representation theorem holds for a {\it solvable} sesquilinear form, i.e. a q-closed form which, up to a perturbation with a bounded form on $\H$, defines a bounded operator, with bounded inverse, acting on the triplet.\\
Moreover, in \cite{Tp_DB} it is shown that densely defined closed sectorial forms are solvable and a criterion of solvability concerning the numerical range of the form is proved.

In this paper we give new results to the approach of \cite{Tp_DB}. In Section \ref{sec:comp_nor}, we analyse the compatible norms with respect to a positive sesquilinear form, which have a key role in the definition of q-closed forms and on the properties of them. In particular, two compatible norms with respect to the same positive sesquilinear form are equivalent if they are defined on the same subspace which is complete with respect to these norms.

In Section \ref{sec:q-cl}, we give the definition of q-closed and q-closable sesquilinear forms and some preliminary properties concerning them. We establish that a densely defined form is q-closable if, and only if, it has a q-closed extension. Moreover, the space $\D[\noo]$ on the domain of a q-closed sesquilinear form with respect to the norm $\noo$ is unique up to isomorphism.

In Section \ref{sec:solv}, we consider solvable sesquilinear forms and the related representation theorem, to which we add new statements. In particular, we prove that the operator associated to a solvable sesquilinear form is the greatest which represents the form and it is self-adjoint if, and only if, the form is symmetric. We also show two examples of q-closed sesquilinear forms which are not solvable.

In Section \ref{sec:cr_solv}, we establish conditions for a q-closed sesquilinear form to be solvable. Some of these conditions are affected by the numerical range of the form and we generalise Theorem 5.11 of \cite{Tp_DB}, considering a bounded form which is not necessary a scalar multiple of the inner product.

In Section \ref{sec:sol_inn}, we focus on solvable sesquilinear forms, defined on $\D$, with respect to an inner product $\pint_\O$, and we find a connection between the associated operator and the operator which represents the form in $\D[\pint_{\O}]$. These results are applied in Section \ref{sec:casi_part}, where we prove that the forms studied in many of the quoted papers are solvable sesquilinear forms. However, these special cases do not exhaust the class of solvable sesquilinear forms.

\section{Compatible norms}
\label{sec:comp_nor}

In this paper, if not otherwise specified, we indicate by $\H$ a Hilbert space, with inner product $\pint$ and norm $\nor$, and by $\D$ a subspace of $\H$. \\

\noindent Let $\O$ be a sesquilinear form defined on $\D$. The {\it adjoint} $\O^*$ of $\O$ is the form on $\D$ given by
$$
\O^*(\xi,\eta)=\ol{\O(\eta,\xi)} \qquad \forall \xi,\eta \in \D.
$$
$\O$ is said to be {\it symmetric} if $\O=\O^*$ and, in particular, $\O$ is {\it positive} if $\O(\xi,\xi)\geq0$ for all $\xi\in\D$.\\
The symmetric forms $\Re \O$ and $\Im \O$, defined by
$$
\Re \O=\frac{1}{2}(\O+\O^*) \qquad \Im \O=\frac{1}{2i}(\O-\O^*),
$$
are called the {\it real part} and the {\it imaginary part} of $\O$, respectively. We have  $\O=\Re\O+i\Im\O$.\\
The {\it numerical range} $\rn_\O$ is the (convex) subset $\{\O(\xi,\xi):\xi\in \D, \n{\xi}=1\}$ of $\C$. We indicate by $N(\O)$ the subspace of $\D$
$$
N(\O):=\{\xi\in \D:\O(\xi,\eta)=0\; \forall \eta \in \D\}.
$$
If $\O$ is a positive sesquilinear form then $
N(\O)=\{\xi\in \D:\O(\xi,\xi)=0\}.	$ \\
We denote by $\iota$ the sesquilinear form which corresponds to the inner product, i.e.  $\iota(\xi,\eta)=\pin{\xi}{\eta}$, with $\xi,\eta \in \H$.\\
Finally, we indicate by $D(T)$ and $Ran(T)$ the domain and the range of an operator $T$ on $\H$, respectively.

\begin{defin}[{\cite[Definition 5.1]{Tp_DB}}]
	\label{norm_comp}
	Let $\T$ be a positive sesquilinear form on $\D$. A norm $\nor_0$ on $\D$ is {\it compatible} with $\T$ if
	\begin{enumerate}
		\item there exists $\alpha>0$ such that $\T(\xi,\xi)\leq \alpha \n{\xi}_0^2$ for all $\xi \in \D$;
		\item if $\{\xi_n\}$ is a sequence on $\D$ such that $\T(\xi_n,\xi_n)\to 0$ and $\n{\xi_n-\xi_m}_0\to 0$, then $\n{\xi_n}_0\to 0$.
	\end{enumerate}
\end{defin}

Note that if $\T$ is a positive sesquilinear form on $\D$ and $\nor_0$ is a compatible norm with $\T$, then $N(\T)=\{0\}$.

Now, let $\T$ be a positive sesquilinear form  on $\D$ such that $N(\T)=\{0\}$, i.e., $\T$ is an inner product on $\D$. To avoid confusion, we denote by $\pint_\T$ the inner product, thus $\pin{\xi}{\eta}_\T=\T(\xi,\eta)$ for all $\xi,\eta \in \D$, while we denote by $\H_\T$ the completion of $\D[\pint_\T]$. Let  $\nor_0$ be a norm on $\D$ such that the operator
\begin{align*}
\mathfrak{I}:\D[\nor_0] &\to \H_\T \\
\xi &\mapsto\;  \xi
\end{align*}
is injective and bounded. Therefore, if $\Ee$ denotes the completion of $\D[\nor_0]$,  $\mathfrak{I}$ extends by continuity to a bounded operator
$\ol{\mathfrak{I}}:\Ee \to \H_\T ,$
with range dense in $\H_\T$. We have the following fact.

\begin{teor}
	\label{pro_equiv_comp}
	$\ol{\mathfrak{I}}$  is injective if, and only if, $\nor_0$ is compatible with $\T$.
\end{teor}
\begin{proof}
	$(\Rightarrow)$
	We only have to prove the condition {\it 2} of Definition  \ref{norm_comp}. Let $\{\xi_n\}$ be a sequence on $\D$ such that $\T(\xi_n,\xi_n)\to 0$ and $\n{\xi_n-\xi_m}_0 \to 0$, thus $\{\xi_n\}$ is a Cauchy sequence on $\D[\nor_0]$, which converges therefore to an element $\xi\in \Ee$. We have
	$
	\ol{\mathfrak{I}} \xi=\ol{\mathfrak{I}} \lim_{n \to \infty} \xi_n = \lim_{n\to \infty} \mathfrak{I} \xi_n=\lim_{n\to \infty}  \xi_n = 0,
	$
	since $\n{\xi_n}_\T=\T(\xi_n,\xi_n)\to 0$.	By the hypothesis,  $\ol{\mathfrak{I}}$ is injective, then $\xi=0$; i.e., $\n{\xi_n}_0\to 0$.\\
	$(\Leftarrow)$
	Suppose that $\ol{\mathfrak{I}}\xi=0$, with $\xi \in \Ee$. Hence, by definition of $\ol{\mathfrak{I}}$, there exists a sequence $\{\xi_n\}$ on $\D$ such that $\n{\xi_n-\xi}_0 \to 0$, and $\mathfrak{I}\xi_n\to 0$; hence, $\T(\xi_n,\xi_n)=\pin{\mathfrak{I}\xi_n}{\mathfrak{I}\xi_n}_\T\to 0$.\\
	But $\nor_0$ is a compatible norm with $\T$, thus $\n{\xi_n}_0\to 0$; that is, $\xi=0$.
\end{proof}

\begin{cor}
	\label{cor_comp}
	Let $\T$ be a positive sesquilinear form on $\D$ with $N(\T)=\{0\}$ and let $\nor_0$ be a norm on $\D$, such that $\T(\xi,\xi)\leq \alpha\n{\xi}_0^2$ for all $\xi \in \D$, and some constant $\alpha>0$. If $\D[\nor_0]$ is complete, then $\nor_0$ is compatible with $\T$.
\end{cor}

Now we prove a lemma similar to the statement of \cite[Exercise 5.10]{Weid}, in which only norms induced by inner products are considered.

\begin{lem}
	Let $E$ be a complex vector space which is a Banach space with respect to two norms $\nor_1$ and $\nor_2$. Suppose that the following conditions hold:
	\begin{enumerate}
		\item if $\{\xi_n\}$ is a sequence on $E$ such that $\n{\xi_n}_1\to 0$ and $\n{\xi_n-\xi_m}_2\to 0$, then $\n{\xi_n}_2\to 0$;
		\item if $\{\xi_n\}$ is a sequence on $E$ such that  $\n{\xi_n}_2\to 0$ and $\n{\xi_n-\xi_m}_1\to 0$, then $\n{\xi_n}_1\to 0$.
	\end{enumerate}
	Then, the norms $\nor_1$ and $\nor_2$ are equivalent.
\end{lem}
\begin{proof}
	We consider the identity $I:E[\nor_1]\to E[\nor_2].$
	$I$ is a closable linear operator, and hence closed, by the  first condition. Conversely, $I^{-1}$  is closable, hence closed,  by the  second condition.
	By the Closed Graph Theorem, $I$ and $I^{-1}$ are two bounded operators, and hence the norms are equivalent.
\end{proof}

The following theorem establishes a condition under which two compatible norms are equivalent.

\begin{teor}
	\label{th_equiv_norm_comp}
	Let $\T$, $\n{\cdot}$ and $\n{\cdot}'$ be a positive sesquilinear form on $\D$, and two norms on $\D$ compatible with $\T$, respectively, and such that $\D[\nor]$ and $\D[\n{\cdot}']$ are Banach spaces. Then the norms $\n{\cdot}$ and $\n{\cdot}'$ are equivalent.
\end{teor}
\begin{proof}
	We prove, using the previous lemma, that $\n{\cdot}$ and $\n{\cdot}'$ are equivalent.\\
	If $\{\xi_n\}$ is a sequence in $\D$ such that $\n{\xi_n}\to 0$ and $\n{\xi_n-\xi_m}'\to 0$, then, by the compatibility of $\n{\cdot}$ with $\T$,
	$	\T(\xi_n,\xi_n)\leq\alpha \n{\xi_n}^2\to 0	$,
	and by the compatibility of $\n{\cdot}'$ with $\T$, we have $\n{\xi_n}'\to 0$.\\
	By a symmetry argument, it is also true that if $\{\xi_n\}$ is a sequence in $\D$ such that $\n{\xi_n}'\to 0$ and $\n{\xi_n-\xi_m}\to 0$, then $\n{\xi_n}\to 0$.
\end{proof}

\begin{cor}
	Let $\T$, $\n{\cdot}$ and $\n{\cdot}'$ be a positive sesquilinear form on $\D$ with $N(\T)=\{0\}$, and two norms on $\D$, respectively. If $\D[\nor]$ and $\D[\nor']$ are Banach spaces, and there exist  $\alpha,\alpha'>0$ such that $\T(\xi,\xi)\leq \alpha \n{\xi}^2$ and $\T(\xi,\xi)\leq \alpha'\n{\xi}'^2$ for all $\xi\in \D$, then the norms $\n{\cdot}$ and $\n{\cdot}'$ are equivalent.
\end{cor}

\section{Q-closed sesquilinear forms}
\label{sec:q-cl}

\no In this section, and for all the rest of the paper, if not otherwise specified, we assume that the subspace $\D$ is dense in  $\H$.

\begin{defin}[{\cite[Definition 5.2]{Tp_DB}}]
	\label{def_q_chiusa}
	Let $\noo$ be a norm on $\D$. A sesquilinear form $\O$ on $\D$ is called {\it  q-closable with respect to}  $\noo$ if $\noo$ is compatible with the inner product $\pint$, i.e.
	\begin{enumerate}
		\item there exists $\alpha>0$ such that $\n{\xi}\leq \alpha \n{\xi}_\O$ for all $\xi \in \D$, i.e. the embedding  $\D[\noo]\to \H$ is continuous;
		\item if $\{\xi_n\}$ is a sequence in $\D$ such that $\n{\xi_n}\to 0$ and $\n{\xi_n-\xi_m}_\O\to 0$, then $\n{\xi_n}_\O\to 0$;
	\end{enumerate}
	and also the following conditions hold
	\begin{enumerate}
		\item[3.] the completion $\Eo$ of $\D[\noo]$ is a reflexive Banach space;
		\item[4.] there exists $\beta >0$ such that $|\O(\xi,\eta)|\leq \beta\n{\xi}_\O\n{\eta}_\O$ for all  $\xi,\eta \in \D$, i.e. $\O$ is bounded on $\D[\noo]$.
	\end{enumerate}
	$\O$ is called {\it q-closed with respect to}  $\noo$ if $\D[\noo]$ is a reflexive Banach space.
\end{defin}

Actually, using Corollary \ref{cor_comp} we see that in the definition of q-closed sesquilinear form the hypothesis {\it 2} is superfluous. Therefore, we can formulate the following proposition.

\begin{pro}
	Let $\noo$ be a norm on $\D$. A sesquilinear form $\O$ on $\D$ is  q-closed with respect to $\noo$ if, and only if, the following statement hold
	\begin{enumerate}
		\item there exists $\alpha>0$ such that $\n{\xi}\leq \alpha \n{\xi}_\O$ for all $\xi \in \D$;
		\item $\D[\noo]$ is a reflexive Banach space;
		\item there exists $\beta >0$ such that $|\O(\xi,\eta)|\leq \beta\n{\xi}_\O\n{\eta}_\O$ for all  $\xi,\eta \in \D$.
	\end{enumerate}
\end{pro}

We show some examples of q-closed or q-closable forms.

\begin{esm}
	A densely defined closed (closable) sectorial form $\O$ on $\D$, with vertex $\gamma\in \R$, is q-closed (q-closable) with respect to the norm $\noo$ defined by
	$	\n{\xi}_\O=(\Re \O(\xi,\xi)+(1-\gamma)\n{\xi}^2)^\mez$, for all $\xi\in \D$ (see \cite[Ch. VI, Theorem 1.11]{Kato}).
\end{esm}

\begin{esm}
	\label{esm_O_T_(2)}
	Let $T$ be a closed operator with domain $\D$, and let $\O_T$ be the sesquilinear form on $\D$ given by
	$\O_T(\xi,\eta)=\pin{T\xi}{\eta}$ for all $\xi,\eta \in \D$.\\
	$\O_T$ is q-closed with respect to the graph norm of $T$,
	$	\n{\xi}_T=(\n{\xi}^2+\n{T\xi}^2)^\mez$ for all $\xi \in \D$.
\end{esm}

\begin{esm}
	\label{esm_for_L2C}
	Let $r:\C\to \C$ be a measurable function and $\O$ the sesquilinear form with domain
	$$
	\D:=\left \{f\in L^2(\C): \int_{\C}|r(x)||f(x)|^2dx< \infty \right\}
	$$
	and given by
	$	\displaystyle
	\O(f,g)=\int_\C r(x)f(x)\ol{g(x)}dx$ for all $\xi,\eta\in \D$.\\
	$\O$ is q-closed with respect the norm
	$$ \n{f}_\O=\left (\int_{\C}(1+|r(x)|)|f(x)|^2dx\right )^{\frac{1}{2}} \qquad \forall f\in \D.$$
\end{esm}

Definition \ref{def_q_chiusa} of q-closable sesquilinear form, actually, is equivalent to request that the form has a q-closed extension, as affirmed in the following proposition.

\begin{pro}
	\label{pro_defi_q_chius}
	A densely defined sesquilinear form is q-closable if, and only if, it admits a q-closed extension.
\end{pro}
\begin{proof}
	$(\Rightarrow)$ This implication is given by \cite[Proposition 5.3]{Tp_DB}. \\
	$(\Leftarrow)$ Let $\O$ be a sesquilinear form on a dense domain  $\D$ in $\H$, and let $\O'$ a q-closed extension, with domain $\D'$. By the hypothesis, there exists a norm $\nor_{\O '}$ on $\D '$, compatible with the inner product $\pint$, such that $\Ee_{\O '} =\D' [\nor_{\O '}]$ is a reflexive Banach space and $\O' $ is  bounded in such space. \\
	Since $\D \subseteq \D '$, we can consider the norm $\noo$ on $\D$, induced by $\nor_{\O '}$ on $\D$. Clearly, $\O$ is q-closable with respect to $\noo$.
\end{proof}

With the aid of the Closed Graph Theorem one can prove the boundedness of a q-closable sesquilinear form defined in the whole space.

\begin{pro}
	Let $\O$ be a q-closable sesquilinear form on $\D$ with respect to $\noo$. If $\D=\H$ then $\O$ is bounded in $\H$.
\end{pro}

The next result shows that, even though Definition \ref{def_q_chiusa} depends on a norm, this norm is uniquely determined up to an equivalence.

\begin{teor}
	\label{th_equiv_norm_q-chius}
	There exists at most one norm (up to equivalence) on $\D$, such that a sesquilinear form on $\D$ is q-closed with respect to it.
\end{teor}
\begin{proof}
	Let $\O$ be a sesquilinear form on $\D$, which is q-closed with respect to two norms $\n{\cdot}_\O$ and $\n{\cdot}_\O'$. Then we have, in particular, that $\n{\cdot}_\O$ and $\n{\cdot}_\O'$ are compatible with the inner product $\pint$ on $\H$, and the spaces $\D[\noo]$, $\D[\noo']$ are complete. The equivalence follows, hence, by Theorem \ref{th_equiv_norm_comp}.
\end{proof}

\section{Solvable sesquilinear forms}
\label{sec:solv}

If $\mathcal{E}$ and  $\mathcal{F}$ are two Banach spaces, we will indicate by $\B(\mathcal{E},\mathcal{F})$ the vector space of all bounded operators from $\mathcal{E}$ into  $\mathcal{F}$, and more simply $\B(\mathcal{E})=\B(\mathcal{E},\mathcal{E})$, if $\mathcal{E}=\mathcal{F}$.\\
We recall that if $\mathcal{E}$ is reflexive and $X\in\B(\mathcal{E},\mathcal{E}^\times) $ then, the {\it adjoint operator} $X^\dagger$ of $X$ is defined by
$	\pin{X^\dagger \xi}{\eta}=\ol{\pin{X\eta}{\xi}}$ for all $ \xi,\eta\in \D,$
and $X^\dagger \in\B(\mathcal{E},\mathcal{E}^\times) $.\\

\noindent Let $\O$ be a q-closed sesquilinear form with respect to a norm $\noo$ on $\D$, a dense subspace of $\H$. We denote by $\Eo=\D[\noo]$ and by $\Eo^\times=\D^\times[\noo^\times]$, the conjugate dual space of $\Eo$.\\
Note that, from the definition of q-closed sesquilinear form, a {\it Banach-Gelfand triplet}
$$\Eo [\noo] \hookrightarrow \H[\nor] \hookrightarrow \Eo ^\times [\noo^\times]$$
is well-defined, see \cite{Tp_DB}. This means that the arrows indicate continuous embeddings with dense range (the Banach-Gelfand triplets are special rigged Hilbert spaces, see also \cite{AITp,ATp}). Hence, $\H$ can be identified with a dense subspace of $\Eo^\times$, and we will indicate the value of a conjugate linear functional $\Lambda$ in an element $\xi \in \Eo$ by $\Lambda(\xi)=\pin{\Lambda}{\xi}$. In other words, we will assume that the form which puts $\Eo$ and $\Eo^\times$ in duality is an extension of the inner product $\pint$ of $\H$.\\

\noindent We denote by $\Po(\O)$ the set of bounded sesquilinear forms  $\Up$ on $\H$, such that
\begin{equation}
\label{cond_1_def_PO}
N(\O+\Up)=\{0\},
\end{equation}
and for all $\L\in \Eo^\times $ there exists $\xi\in \Eo$ such that
\begin{equation}
\label{cond_2_def_PO}
\pin{\L}{\eta}=(\O+\Up)(\xi,\eta) \qquad \forall \eta \in \Eo .
\end{equation}

\begin{defin}
	A q-closed form $\O$ is said to be  {\it solvable with respect to}  $\noo$ if the set $\Po(\O)$ is not empty (see \cite[Definition 5.5]{Tp_DB}).
\end{defin}

\begin{osr}
	A densely defined closed sectorial form is solvable (see \cite[Example 5.8]{Tp_DB}).
\end{osr}

Let $\Up$ be a bounded sesquilinear form on $\H$, we put $\O_\Up:=\O+\Up$. If $\xi\in \D$, we can define the conjugate linear functional $\O_\Up^\xi$ on $\Eo$ by
$$
\pin{\O_\Up^\xi}{\eta}=\O_\Up(\xi,\eta)=\O(\xi,\eta)+\Up(\xi,\eta) \qquad \forall \eta \in \Eo,
$$
which is bounded in $\Eo$, and also the operator
\begin{align*}
X_\Up: \Eo &\to \Eo^\times \\
\xi &\mapsto \O_\Up^\xi,
\end{align*}
is bounded; i.e., $X_\Up \in \B(\Eo,\Eo^\times)$.\\

\noindent The next characterization of forms belonging to the set $\Po(\O)$ holds.

\begin{lem}[{\cite[Lemma 5.6]{Tp_DB}}]
	\label{lem_biiez}
	Let $\O$ be a q-closed sesquilinear form on $\D$ with respect to $\noo$. Then, $\Up \in \Po(\O)$ if, and only if, $X_\Up$ is a bijection of $\Eo$ onto $\Eo^\times$ if, and only if, $X_\Up$ is invertible with bounded inverse.
\end{lem}

The following is the converse of Theorem \ref{th_equiv_norm_q-chius} and, as shown in the next corollary, a relevant consequence of the norm equivalence discussed in Section \ref{sec:comp_nor} concerns the solvability of a q-closed sequilinear form.

\begin{teor}
	\label{th_q_cl_sol_norm_eq}	
	Let $\O$ be a q-closed sesquilinear form on $\D$ with respect to $\noo$ and let $\noo'$ be a norm equivalent to $\noo$. Then, $\O$ is q-closed with respect to $\noo'$. If, moreover, $\O$ is solvable with respect to $\noo$, then $\O$ is solvable with respect to $\noo'$.
\end{teor}
\begin{proof}
	It is easy to prove that $\O$ is q-closed with respect to $\noo'$. We suppose that $\O$ is solvable with respect to $\noo$, and denote by $\Eo$ and $\Eo'$ the Banach spaces $\D[\nor_\O]$ and $\D[\nor_\O']$, respectively.\\
	Then, by the hypothesis, there exists a bounded sesquilinear form $\Up$ on $\H$ such that the operator $X_\Up$ defined by
	\begin{align*}
	X_\Up:\Eo&\to \Eo^\times  \\
	\xi &\mapsto  \O_\Up^\xi
	\end{align*}
	where	$	\O_\Up^\xi(\eta)=\O(\xi,\eta)+\Up(\xi,\eta)$ for all $\eta \in \Eo,	$ is bijective and continuous. \\
	Due to the assumptions, there exists a continuous isomorphism  $I:\Eo'\to\Eo$  of $\Eo'$ onto $\Eo$. Then, $I^\times:\Eo^\times\to\Eo'^\times$ is a continuous isomorphism between the dual spaces $\Eo^\times$ and $\Eo'^\times$. Therefore the operator $X_\Up':=I^\times X_\Up I:\Eo' \to \Eo'^\times$ is an isomorphism and
	$
	X_\Up' \xi =I^\times X_\Up I\xi=I^\times X_\Up \xi = I^\times \O_\Up^\xi=\O_\Up^\xi,
	$
	so $\O$ is solvable with respect to $\noo'$.
\end{proof}

\begin{cor}
	\label{cor_q_chi_non_ris}
	If $\O$ is a q-closed, non-solvable, sesquilinear form on $\D$ with respect to a norm $\nor_\O$, then $\O$ is not solvable with respect to any norm.
\end{cor}
\begin{proof}
	Let $\nor_\O'$ be a norm with respect to which $\O$ is q-closed. By Theorem \ref{th_equiv_norm_q-chius} the norms $\nor_\O$ and $\nor_\O'$ are equivalent, hence by Theorem \ref{th_q_cl_sol_norm_eq} $\O$ is not solvable with respect to $\noo'$.
\end{proof}

Now, we recall Theorem 5.9 of \cite{Tp_DB}, which generalises Kato's First Representation Theorem, and add new properties of the operator constructed in the proof of that theorem.

\begin{teor}
	\label{th_rapp_risol}
	Let $\O$ be a solvable sesquilinear form on $\D$ with respect to a norm $\noo$. Then there exists a closed operator $T$, with dense domain $D(T)\sub \D$ in $\H$, such that
	\begin{equation}
	\label{eq_rapp}
	\O(\xi,\eta)=\pin{T\xi}{\eta} \qquad \forall \xi\in D(T),\eta \in \D.
	\end{equation}
	Moreover,
	\begin{enumerate}
		\item $D(T)$ is dense in $\D[\noo]$;
		\item if $T'$ is an operator $\H$ with domain $D(T')\subseteq \D$ and
		\begin{equation}
		\label{th_rapp_T'}
		\O(\xi,\eta)=\pin{T'\xi}{\eta}
		\end{equation}
		for all $\xi\in D(T')$ and $\eta$ which belongs to a dense subset of $\D[\noo]$, then $T' \subseteq T$;
		\item if $\Up \in \Po(\O)$ and $B\in \B(\H)$ is the bounded operator such that
		$\Up(\xi,\eta)=\pin{B\xi}{\eta}$  for all $\xi, \eta \in \H$,
		then $T+B$ is invertible and $(T+B)^{-1}\in \B(\H)$. In particular, if $\Up=-\lambda \iota$, with $\lambda \in \C$, then $\lambda \in \rho(T)$, the resolvent set of  $T$;
		\item $T$ is the unique operator satisfying \emph{(\ref{eq_rapp})} with the property that $T+B$ has range $\H$.
	\end{enumerate}
\end{teor}
\begin{proof}
	Let $\Eo=\D[\noo]$, $\Up\in \Po(\O)$, $B$ the bounded operator associated to $\Up$ and $X_\Up$  as above. For the proof of the existence of $T$ see \cite{Tp_DB}. We recall that  $X_\Up$ has a bounded inverse $X_\Up^{-1}$, $T=S-B$,
	$
	D(T)=D(S)=X_\Up ^{-1}(\H)=\{\xi\in \D: X_\Up \xi \in \H\},
	$
	and $S\xi = X_\Up \xi$ for all $\xi \in \D$.\\	
	As shown in the proof of Theorem 5.9 of \cite{Tp_DB}, $D(T)$ is dense in $\D[\nor_\O]$; i.e., the point {\it 1}.\\
	Now, note that the relation (\ref{th_rapp_T'}) extends by continuity to all $\eta\in \D$. Let $S'$ be the operator in $\H$ with domain $D(S')=D(T')$ and defined by $S'=T'+B$. Then
	$$
	\l S'\xi | \eta \r = \l T' \xi |\eta \r + \l B \xi | \eta \r= \O(\xi, \eta)+\Upsilon(\xi,\eta)=(\O+\Upsilon)(\xi,\eta)=\l X_\Upsilon \xi | \eta \r
	$$
	for all $\xi \in D(S'),\eta \in \Eo$. Hence if $\xi \in D(S')$ we have $X_\Upsilon \xi = S' \xi \in \H$, and by definition, $\xi \in D(S)$ and $S \xi = X_\Upsilon \xi$; i.e., $S'\subseteq S$ and $T' \subseteq T$. This proves the point {\it 2} of the statement. \\
	For the third part of the statement, by the boundedness of $X_\Up^{-1}$, there exists $M>0$ such that
	$\n{X_\Up^{-1} \Lambda}_\O\leq M\n{\Lambda}_\O^\times$ for all $\Lambda \in \Eo^\times$,
	and, since $S^{-1}$ is the restriction of $X_\Up^{-1}$ to $\H$,
	$$
	\n{S^{-1}\eta}\leq \alpha \n{S^{-1}\eta}_\O \leq \alpha M \n{\eta}_\O^\times \leq \alpha M \n{\eta} \qquad \forall \eta \in \H,
	$$
	i.e., $S^{-1}$ is bounded in $\H$. But, by definition, $S=T+B$.\\
	Finally, if $T'$ is an operator satisfying (\ref{eq_rapp}) and $T'+B$ has range $\H$. As we have already proved, $T'\sub T$; hence $S'=T'+B$ is such that $S'=S$, because $S$ is bijective. Therefore, $T'=T$.
\end{proof}

\begin{osr}
	The operator $T$ constructed above does not depend on the particular considered norm $\noo$, since the norm is unique up to equivalence, and does not depend even on the particular chosen bounded form $\Upsilon\in \Po(\O)$. Indeed, if $T'$ is the operator constructed in the same way as $T$ (with domain $D(T')$) considering another norm $\noo'$ with respect to which $\O$ is q-closed (or considering another form $\Upsilon\in \Po(\O)$), then by the previous theorem we have
	$$
	\O(\xi,\eta)=\pin{T\xi}{\eta} \qquad \forall \xi\in D(T),\eta \in \D,
	$$
	and
	$$
	\O(\xi,\eta)=\pin{T'\xi}{\eta} \qquad \forall \xi\in D(T'),\eta \in \D.
	$$
	So, applying the second statement of Theorem \ref{th_rapp_risol} two times,  $T=T'$.
\end{osr}

For the characteristics of the obtained operator, we give the following definition.

\begin{defin}
	Let $\O$ be a solvable sesquilinear form. The operator $T$ in  Theorem \ref{th_rapp_risol} is called the {\it operator associated} to $\O$.
\end{defin}

\begin{osr}
	As we will see in the examples of the next sections, the property of uniqueness established in the fourth statement of Theorem  \ref{th_rapp_risol} is very useful to determine the operator associated to a solvable sesquilinear form.
\end{osr}

\begin{cor}
	\label{cor_th_rapp_est}
	Let $\O$ be a solvable sesquilinear form on $\D$ with respect $\noo$ and let $T$ be its associated operator. If $\xi \in \D, \chi \in \H$ and $\O(\xi,\eta)=\pin{\chi}{\eta}$ for all $\eta$ which belongs to a dense subspace in $\D[\noo]$, then $\xi\in D(T)$ and $T\xi=\chi$.
\end{cor}

The next theorem shows that for a sesquilinear form the property of being q-closed, or solvable, is preserved by passing to the adjoint.

\begin{teor}
	\label{th_ris_agg}
	If $\O$ is a q-closed sesquilinear form on $\D$ with respect to a norm $\noo$, then also the adjoint  $\O^*$ is q-closed with respect to $\noo$. Moreover $\Up\in \Po(\O)$ if, and only if, $\Up^*\in \Po(\O^*)$, and if $\O$ is solvable with respect to  $\noo$, and with the associated operator $T$, then $\O^*$ is solvable with respect to $\noo$, with associated operator $T^*$.
\end{teor}
\begin{proof}
	The first statement is clear, while for the second part, by a symmetry argument, it is sufficient to prove that if $\Up\in \Po(\O)$ then $\Up^*\in \Po(\O^*)$.\\
	Assume that $\O$ is solvable with respect to $\noo$ and let $\Up\in \Po(\O)$. By Lemma \ref{lem_biiez} and, setting $\Eo=\D[\noo]$, we have that $X_\Up$ is bijective and consequently the adjoint $X_\Up^\dagger:\Eo\to \Eo^\times$ is bijective. From
	\begin{equation*}
	\pin{X_\Up^\dagger \xi}{\eta} = \ol{\pin{X_\Up \eta}{\xi}} = \O^*(\xi,\eta)+\Up^*(\xi,\eta) \qquad \forall \xi,\eta \in \D
	\end{equation*}
	it follows that $\O^*$ is solvable with respect to $\noo$ and $\Up^*\in \Po(\O^*)$.	As proved in the proof of Theorem \ref{th_rapp_risol}, $T=S-B$, with $S=X_{\Up|D(T)}$ the restriction to $D(T)$ of $X_\Up$,
	and $B$ is the bounded operator in $\H$ such that
	$\Up(\xi,\eta)=\pin{B\xi}{\eta}$ for all  $\xi,\eta \in \H.$\\
	Let $T'$ be the associated operator to $\O^*$, then we have that $T'=S'-B'$, with $S'=X^\dagger_{\Up|D(T')}$ the restriction to $D(T')$ of $X^\dagger_\Up$, and $B'$ is the bounded operator in $\H$ such that
	$
	\Up^*(\eta,\xi)=\pin{B'\eta}{\xi}$ for all $\eta,\xi \in \H.
	$ 	Hence $B'=B^*$, but $D(S')={X_\Up^\dagger}^{-1}(\H)=D(S^*)$ and $S'=X^\dagger_{\Up|D(S^*)}=S^*$ (see the proof of \cite[Theorem 5.9]{Tp_DB}). Hence, taking into account that $B$ is a bounded operator in $\H$, we conclude that
	\[
	T'=S'-B'=S^*-B^*=(S-B)^*=T^*.  \qedhere
	\]
\end{proof}

An immediate consequence of the previous theorem  concerns symmetric forms.

\begin{cor}
	\label{cor_auto}
	The operator associated to a solvable symmetric form $\O$ is self-adjoint, and $\Up\in \Po(\O)$ if, and only if, $\Up^*\in \Po(\O)$.
\end{cor}

Using point {\it 1} in Theorem \ref{th_rapp_risol}, we prove a connection between the numerical ranges of a solvable sesquilinear form and its associated operator, which implies also the converse of Corollary \ref{cor_auto}.

\begin{pro}
	The numerical range of the operator associated to a solvable sesquilinear form is a dense subset of the numerical range of the form.
\end{pro}

\begin{cor}
	\label{cor_auto<->simm}
	The operator associated to a solvable sesquilinear form is self-adjoint if, and only if, the form is symmetric.
\end{cor}

The following interesting question arises: are there two solvable sesquilinear forms with the same associated operator? The answer is affirmative if the domain of a form is contained in the one of the other form. This condition is already considered in Proposition 3.2 of \cite{FHdeSW} (see also Theorem 3.2 of \cite{GKMV}). We will come back to this problem in Section \ref{sec:casi_part}.

\begin{lem}
	\label{lem_unic_op}
	Let $\O_1$ and $\O_2$ be two solvable sesquilinear forms with domains $\D_1$ and $\D_2$, respectively, and with the same associated operator $T$. If $\D_1\subseteq \D_2$ then $\D_1=\D_2$ and $\O_1=\O_2$.
\end{lem}
\begin{proof}
	First of all, we will prove that $\D_1=\D_2$. Denote by $\nor_1$ a norm with respect to which $\O_1$ is solvable and by $\nor_2$ a norm with respect to which $\O_2$ is solvable. Put $\Ee_1:=\D_1[\nor_1]$ and $\Ee_2:=\D_2[\nor_2]$. \\
	The embedding $\mathfrak{I} : \Ee_1 \to \Ee_2$ is closed. Indeed, if $\{\xi_n\}$ is a sequence in $\D_1$ such that $\xi_n \conv{\nor_1} \xi$ and $\xi_n=\mathfrak{I}\xi_n \conv{\nor_2} \xi'$, for some $\xi\in \D_1, \xi'\in \D_2$, then $\xi_n \to \xi$ and $\xi_n \to \xi'$ in $\H$, and so $\mathfrak{I}\xi=\xi=\xi'$. By the Closed Graph Theorem, $\mathfrak{I}$ is bounded, i.e. there exists $\alpha>0$ such that $\n{\xi}_2\leq \alpha\n{\xi}_1$ for all $\xi\in \D_1$.\\
	Let $\xi\in \D_2$, then the functional $\Lambda(\eta)=\O_2(\xi,\eta)$, with $\eta \in \D_1$, is an element of $\Ee_1^\times$. Since by the hypothesis $\O_1$ is solvable, there exists a bounded form $\Up\in \Po(\O)$, and for every $\xi\in \D_2$ there exists $\chi \in \D_1$ such that
	$$
	\O_2(\xi,\eta)=\Lambda (\eta)=(\O_1+\Up)(\chi,\eta) \qquad \forall \eta \in \D_1,
	$$
	and hence
	\begin{equation}
	\label{pass_D1cD2}
	(\O_1^*+\Up^*)(\eta,\chi)=\O_2^*(\eta,\xi) \qquad \forall \eta \in \D_1.
	\end{equation}
	By Theorem \ref{th_ris_agg}, $T^*$ is the operator associated with both $\O_1^*$ and  $\O_2^*$. Then, for all $\eta \in D(T^*)$, we have from (\ref{pass_D1cD2}) that $	\pin{(T^*+B^*)\eta}{\chi}=\pin{T^*\eta}{\xi}$,
	and therefore $\pin{T^*\eta}{\chi-\xi}=\pin{-B^*\eta}{\chi}=\pin{\eta}{-B\chi} $  for all $\eta \in D(T^*)$.
	By definition of the adjoint of $T^*$, $\chi-\xi\in D(T)\sub\D_1$ and $T(\chi-\xi)=-B\chi$. Hence, $\xi\in \D_1$, and $\D_1=\D_2$.\\
	The domain $D(T)$ of $T$ is dense in both $\D[\nor_1]$ and $\D[\nor_2]$, and on $D(T)$ the forms coincide. By Theorem \ref{th_equiv_norm_q-chius}, the two norms are equivalent, hence the equality of $\O_1$ and $\O_2$ is true, by continuity, in the whole of $\D$.
\end{proof}

We end this section with an example of a solvable sesquilinear form, and two examples of q-closed sesquilinear forms which are not solvable.

\begin{esm}
	\label{esm_form_reg_l_2(3)}
	As proved in \cite[Example 6.1]{Tp_DB}, for every sequence $\ull{\alpha}:=\{\alpha_n\}$ of complex numbers, the form
	$$
	\Oa(\{\xi_n\},\{\eta_n\})=\sum_{n=1}^\infty \alpha_n \xi_n \ol{\eta_n}
	$$
	with domain $D(\Oa)=\{\{\xi_n\}\in \ll_2:\sum_{n=1}^\infty |\alpha_n||\xi_n|^2< \infty\}$,
	is solvable with respect to the norm given by
	$$
	\n{\{\xi_n\}}_{\Oa}= \left (\sum_{n=1}^\infty |\xi_n|^2+\sum_{n=1}^\infty |\alpha_n| |\xi_n|^2\right )^\mez.
	$$
	In particular, we have that $\Up=-\lambda \iota$, with $\lambda\notin \ol{\{\alpha_n:n\in \N\}}$, belongs to $\Po(\Oa)$.\\
	We will determine the operator $T$ associated to $\Oa$. We denote by $\pint$ the inner product of $\ll_2$, and by $\Ma$ the multiplication operator by $\ull{\alpha}$, with domain
	$$
	D(\Ma)=\left \{\{\xi_n\}\in \mathit{l}_2 : \sum_{n=1}^\infty |\alpha_n|^2|\xi_n|^2 <\infty \right \}
	$$
	and given by $\Ma \{\xi_n\}=\{\alpha_n\xi_n\},$ for every $\{\xi_n\} \in D(\Ma)$.\\
	If $\{\xi_n\} \in D(\Ma)$, then
	$$
	2\sum_{n=1}^\infty |\alpha_n||\xi_n|^2\leq \sum_{n=1}^\infty |\xi_n|^2+\sum_{n=1}^\infty |\alpha_n|^2|\xi_n|^2<\infty
	$$
	hence  $\{\xi_n\} \in D(\Oa)$, and, moreover,
	$$
	\Oa(\{\xi_n\},\{\eta_n\})=\pin{\mathscr{M}_{\ull{\alpha}} \{\xi_n\}}{\{\eta_n\}} \qquad \forall \{\xi_n\}\in D(\Ma),\{\eta_n\}\in D(\Oa).
	$$
	By Theorem \ref{th_rapp_risol} we have $\mathscr{M}_{\ull{\alpha}}\sub T$. \\
	If $\{\xi_n\} \in D(T)\sub D(\Oa)$ and $T\{\xi_n\}=\{\chi_n\}$, then, in particular,
	$$
	\Oa(\{\xi_n\},e_n)=\pin{T\{\xi_n\}}{e_n} \qquad \forall n\in \N,
	$$
	where $\{e_n\}$ is the canonical basis of $\ll_2$; i.e., $	\chi_n=\alpha_n\xi_n$ for all $ n\in \N$,
	and therefore $T\{\xi_n\}=\{\alpha_n\xi_n\}$. But, taking into account that $\{\chi_n\}\in \ll_2$, we have $\displaystyle \sum_{n=1}^\infty |\alpha_n|^2|\xi_n|^2<\infty$; i.e., $\{\xi_n\}\in D(\Ma)$. Hence, $T=\Ma$.
\end{esm}

\begin{esm}[Example \ref{esm_O_T_(2)}]
	Let $T$ be a closed operator with dense domain $\D$ in $\H$, and let $\O_T$ be the sesquilinear form on $\D$ given by	$\O_T(\xi,\eta)=\pin{T\xi}{\eta}$ for all $\xi,\eta \in \D$,
	which, as seen in the Example \ref{esm_O_T_(2)}, is q-closed with respect to the graph norm $\nor_T$ of $T$.\\
	$\O_T$ is solvable with respect to $\nor_T$ if, and only if, $\D=\H$ and $T\in \B(\H)$. Indeed, if $\Up\in \Po(\O)$ and $B\in \B(\H)$ such that $\Up(\xi,\eta)=\pin{B\xi}{\eta}$ for all $\xi,\eta \in \D$, then for all $\Lambda \in \D^\times [\n{\cdot}_T^\times]$ there exists $\xi \in \D$ such that
	$\pin{\Lambda }{\eta}=\pin{(T+B)\xi}{\eta}$ for all $\eta \in \D,$
	and since
	$$\displaystyle \sup_{\n{\eta}=1} |\pin{(T+B)\xi}{\eta}|\leq \n{(T+B)\xi},$$
	we have $\Lambda \in \H$; i.e., $\D^\times=\H$. \\
	By the Closed Graph Theorem, the norms of $\D^\times$ and of $\H$ are equivalent, and hence also the norms $\nor_T$ and $\nor$ are equivalent. It follows that $T$ is bounded and, since it is closed, it has domain $\D=\ol{\D}=\H$; i.e., $T\in \B(\H)$.\\
	The form $\O_T$, with $T$ unbounded, is therefore q-closed and non-solvable with respect to $\nor_T$, and by Corollary \ref{cor_q_chi_non_ris}, it is not solvable with respect to any norm.\\
	Instead, if $\D=\H$ and $T\in \B(\H)$ then $\O_T$ is a bounded sesquilinear form in $\H$ with domain $\H$. So, in particular, it is a sectorial form and then solvable.
\end{esm}

Note that this example depends strongly on the fact that the domain $\D$ of $\O_T$ is exactly the domain of $T$ and on the choice of the norm $\nor_T$, which makes $\O_T$ q-closed. Although $\O_T$ is not solvable (if $T$ is unbounded), it may admit a solvable extension. Indeed, for example, if $T$ is a sectorial operator defined on a dense domain $\D$ in $\H$ then, from \cite[Ch. V, Theorem 1.27]{Kato}, $\O_T$ is closable, hence it is q-closable with respect to a certain norm (which is not, in general, a norm equivalent to $\nor_T$), and its closure $\ol{\O_T}$ (which, in general, is defined on a bigger domain than $\D$) is a solvable form.

Unlike the previous example, in the next one we consider a norm which is not induced by an inner product.

\begin{esm}
	Let $\H=L^2([0,1])$, with the usual inner product $\pint_2$ and the norm $\nor_2$, and let $p>2$. The reflexive Banach space $L^p([0,1])$, hence, is continuously embedded in $\H$, and we denote by $\D$ the (dense) subspace of $\H$ which corresponds to the range of this embedding. We consider a measurable function  $w\in L^{\frac{p}{p-2}}([0,1])$. Then the sesquilinear form  $\O$ on $\D$ given by
	$$
	\O(f,g)=\int_0^1 wf\ol{g}dx \qquad \forall f,g\in \D,
	$$
	is well-defined. Indeed, by the H\"{o}lder inequality, for all $f,g\in L^p([0,1])$ one has
	$$
	\int_0^1 |wf\ol{g}|dx\leq \n{w}_{\frac{p}{p-2}}\n{f}_p\n{g}_p <\infty,
	$$
	where $\n{\cdot}_{\frac{p}{p-2}}$ and $\n{\cdot}_p$ are the norms of $L^{\frac{p}{p-2}}([0,1])$ and $L^p([0,1])$, respectively.\\
	We continue to denote by $\nor_p$ the norm on $\D$ such that $\D[\nor_p]\equiv L^p([0,1])$. $\O$ is q-closed with respect to $\nor_p$, but $\O$ is not solvable with respect to $\nor_p$. Were it so then $\D[\nor_p]$ would be isomorphic (by the operator $X_\Up$, with $\Up$ a bounded sesquilinear form) to its dual $\D^\times[\nor_p^\times]$. But $L^p([0,1])$ is not isomorphic to its dual.	Corollary \ref{cor_q_chi_non_ris} establishes hence that $\O$ is not solvable with respect to any norm.
\end{esm}

\section{Criteria of solvability}
\label{sec:cr_solv}

In this section, we formulate criteria for establishing if a given q-closed sesquilinear form is solvable, in addition to those already seen, like Lemma \ref{lem_biiez}.
The definition of q-closed sesquilinear form is not affected by the notion of the numerical range, but the latter will play a relevant role in some criteria.

\begin{lem}
	\label{lem_crit}
	Let $\O$ be a q-closed sesquilinear form on $\D$ with respect to a norm $\noo$ and let $\Up$ be a bounded form on $\H$. Consider the following statements:
	\begin{enumerate}
		\item[(a)] $N(\O+\Up)=\{0\}$;
		\item[(b)] $N(\O^*+\Up^*)=\{0\}$;		
		\item[(c)] there exists a constant $c_1>0$ such that
		$$
		c_1\n{\xi}_\O\leq \sup_{\n{\eta}_\O=1} |(\O+\Up)(\xi,\eta)| \qquad \forall \xi \in \D;
		$$
		\item[(d)] there exists a constant $c_2>0$ such that
		$$
		c_2\n{\eta}_\O\leq \sup_{\n{\xi}_\O=1} |(\O+\Up)(\xi,\eta)| \qquad \forall \eta \in \D;
		$$
		\item[(e)] for every sequence $\{\xi_n\}$ in $\D$ such that
		$$ \sup_{\n{\eta}_\O=1} |(\O+\Up)(\xi_n,\eta)|\to 0,$$
		$ \n{\xi_n}_\O\to 0$ results;
		\item[(f)] for every sequence $\{\eta_n\}$ in $\D$ such that
		$$ \sup_{\n{\xi}_\O=1} |(\O+\Up)(\xi,\eta_n)|\to 0,$$
		$ \n{\eta_n}_\O\to 0$  results.
	\end{enumerate}
	Statements (c) and (e) are equivalent, and the same holds for (d) and (f). Moreover, the following statements are equivalent.
	\begin{enumerate}
		\item $\Up\in \Po(\O)$;
		\item (a) and (d) hold;
		\item (b) and (c) hold;
		\item (c) and (d) hold.
	\end{enumerate}
\end{lem}
\begin{proof}
	The equivalences are proved noting the following facts.
	\begin{itemize}
		\item $N(\O+\Up)=\{0\}$ if, and only if, $X_\Up$ is invertible, if, and only if, $X_\Up^\dagger$ has dense range (by the reflexivity of $\Eo=\D[\noo]$).
		\item $N(\O^*+\Up^*)=\{0\}$ if, and only if, $X_\Up^\dagger$ is invertible, if, and only if, $X_\Up$ has dense range.
		\item Since
		$$
		\sup_{\n{\eta}_\O=1} |(\O+\Up)(\xi,\eta)|=\sup_{\n{\eta}_\O=1} |\pin{X_\Up \xi}{\eta}|= \n{X_\Up \xi}_\O^\times \qquad \forall \xi \in \D,
		$$
		statement (c) holds if, and only if, $X_\Up$ is invertible with bounded inverse, if, and only if, statement (e) holds.
		\item Since
		$$
		\sup_{\n{\xi}_\O=1} |(\O+\Up)(\xi,\eta)|=\sup_{\n{\xi}_\O=1} |\pin{X_\Up^\dagger \eta}{\xi}|= \n{X_\Up^\dagger \eta}_\O^\times \qquad \forall \eta \in \D,
		$$
		statement (d) holds if, and only if, $X_\Up^\dagger$ is invertible with bounded inverse, if, and only if, statement (f) holds.
		\item $X_\Up$ is a bijection of $\Eo$ onto $\Eo^\times$ if, and only if, $X_\Up^\dagger$ is a bijection of $\Eo$ onto $\Eo^\times$.	\qedhere	
	\end{itemize}
\end{proof}

Using Lemma \ref{lem_crit} we obtain two specific cases.

\begin{teor}
	\label{crit_gener_rn}
	Let $\O$ be a q-closed sesquilinear form on $\D$ with respect to a norm $\noo$ with numerical range $\rn_\O$ and let $\Up$ be a  bounded form in $\H$. Assume that $\rn_\O\cap \rn_{-\Up}=\varnothing$, where $\rn_{-\Up}$ is the numerical range of $-\Up$. Then, $\Upsilon \in \Po(\O)$ if, and only if, either the statement 1. or 2. below holds
	\begin{enumerate}
		\item if $\{\xi_n\}$ is a sequence in $\D$ such that $\displaystyle \sup_{\n{\eta}_\O=1} |(\O+\Up)(\xi_n,\eta)|\to 0$, then $ \n{\xi_n}_\O\to 0$.
		\item there exists a constant $c>0$ such that
		$$ c\n{\xi}_\O\leq \sup_{\n{\eta}_\O=1} |(\O+\Up)(\xi,\eta)| \qquad \forall \xi \in \D.$$
	\end{enumerate}
\end{teor}
\begin{proof}
	We prove that $N(\O^*+\Up^*)=\{0\}$, so the conclusion follows by Lemma \ref{lem_crit}, noting that statements {\it 1} and {\it 2} are the (e) and the (c) in this lemma.\\
	If $\eta\in N(\O^*+\Up^*)$ then $0=(\O^*+\Up^*)(\eta,\xi)=\ol{(\O+\Up)(\xi,\eta)}$
	for all $\xi \in \D$. In particular, $\O(\eta,\eta)=-\Up(\eta,\eta)$ and therefore, from $\rn_\O\cap \rn_{-\Up}=\varnothing$, $\eta=0$.
\end{proof}

\begin{cor}
	\label{cor_crit_qc'}
	Let $\O$ be a q-closed sesquilinear form on $\D$ with respect to a norm $\noo$ with numerical range $\rn_\O$ and let $\lambda\notin \rn_\O$. Then $-\lambda \iota \in \Po(\O)$ if, and only if, one of the statements 1. or 2. in Theorem \ref{crit_gener_rn} holds with $\Up=-\lambda \iota$.
\end{cor}

The next criterion is a generalization of \cite[Theorem 5.11]{Tp_DB}.

\begin{teor}
	\label{th_cond_qc}
	Let $\O$ be a q-closed sesquilinear form on $\D$ with respect to a norm $\noo$ with numerical range $\rn_\O$. Suppose that the following condition is satisfied
	\begin{enumerate}
		\item[\emph{(qc)}] if $\{\xi_n\}$ is a sequence in $\D$ such that $\displaystyle \lim_{n\to \infty }\n{\xi_n}= 0$ and $\displaystyle\lim_{n\to \infty} |\O(\xi_n,\xi_n)|= 0$, then $\displaystyle\lim_{n\to \infty} \n{\xi_n}_\O=0$.
	\end{enumerate}
	Let, moreover, $\Up$ be a bounded sesquilinear form in $\H$, and let $\rn_{-\Up}$ be the numerical range of  $-\Up$. If $\ol{\rn_\O}\cap \ol{\rn_{-\Up}}=\varnothing$, then $\Up\in \Po(\O)$, hence $\O$ is solvable with respect to $\noo$. In particular, if $\lambda\notin \ol{\rn_\O}$, then $-\lambda \iota \in \Po(\O)$ and $\O$ is solvable with respect to $\noo$.
\end{teor}
\begin{proof}
	The hypotheses allow to apply Theorem \ref{crit_gener_rn}. We prove that the condition {\it 1} of that theorem holds.\\
	Let $\{\xi_n\}$ be a sequence in $\D$ such that $\displaystyle \sup_{\n{\eta}_\O=1} |(\O+\Up)(\xi_n,\eta)|\to 0$. We will show that $\n{\xi_n}_\O\to 0$. \\
	Let $\{\xi_{n_k}\}$ be the subsequence of $\{\xi_n\}$ which consists of all non zero elements.	If $\{\xi_{n_k}\}$ is a finite set then obviously $\n{\xi_n}_\O\to 0$. Otherwise, setting $\phi_{n_k}=\frac{\xi_{n_k}}{\n{\xi_{n_k}}}$, we have in particular
	\begin{equation}
	\label{pass_qc}
	|(\O+\Up)(\xi_{n_k},\phi_{n_k})|\to 0,	
	\end{equation}
	i.e.,
	$\n{\xi_{n_k}}|(\O+\Up)(\phi_{n_k},\phi_{n_k})|\to 0.$\\
	From $\ol{\rn_\O}\cap \ol{\rn_{-\Up}}=\varnothing$ and taking into account that $\n{\phi_{n_k}}=1$, there exists $d>0$ such that $d\leq |(\O+\Up)(\phi_{n_k},\phi_{n_k})|$ for all $k\in \N$. Then
	$$
	d\n{\xi_{n_k}}\leq \n{\xi_{n_k}}|(\O+\Up)(\phi_{n_k},\phi_{n_k})|\to 0.
	$$
	Moreover, from (\ref{pass_qc}), $|(\O+\Up)(\xi_{n_k},\xi_{n_k})|\to 0$, but $\Up(\xi_{n_k},\xi_{n_k})\to 0$ because $\Up$ is bounded and $\xi_{n_k}\to 0$, therefore $|\O(\xi_{n_k},\xi_{n_k})|\to 0$.\\	
	The condition (qc) implies that $\n{\xi_{n_k}}_\O\to 0$, and hence $\n{\xi_n}_\O\to 0$.
\end{proof}

This criterion applies, in particular, to closed sectorial forms, as said in \cite[Example 5.12]{Tp_DB}. But as shown in the following example, not every solvable sesquilinear form (whose numerical range is different from the whole complex plane) satisfies the condition (qc).

\begin{esm}[Example \ref{esm_form_reg_l_2(3)}]
	\label{esm_form_reg_l_2(4)}
	As we have seen, the form $\Oa$ for all sequences $\ull{\alpha}:=\{\alpha_n\}$ of complex numbers, is solvable with respect to the norm
	$$
	\n{\{\xi_n\}}_{\Oa}= \left (\sum_{n=1}^\infty |\xi_n|^2+\sum_{n=1}^\infty |\alpha_n| |\xi_n|^2\right )^\mez,
	$$
	and $\Up=-\lambda \iota\in \Po(\Oa)$, with $\lambda\notin \ol{\{\alpha_n:n\in \N\}}$.
	We show that, in general, the condition (qc) is not satisfied. Indeed, we suppose, for instance,  $\ull{\alpha}:=\{(-1)^nn\}$ (in such case $\Oa$ has numerical range $\R$) and let $\zeta_n:=\frac{1}{\sqrt{2n}}e_{2n}+\frac{1}{\sqrt{2n+1}}e_{2n+1}\in D(\Oa)$, where $\{e_n\}$ is the canonical basis of $\ll_2$. Then $\n{\zeta_n}\to 0$ and $\Oa(\zeta_n,\zeta_n)=0$,
	but $\n{\zeta_n}_{\Oa}$
	does not converge to $0$.\\
	It is possible to use also Corollary \ref{cor_crit_qc'} to prove that $\Up=-\lambda \iota$, with $\lambda\notin \ol{\{\alpha_n:n\in \N\}}$, belongs to $\Po(\Oa)$.
	Indeed, let $\xi\in \D$, so that $\xi=\{\xi_n\}$ and put $\chi:=\left \{\frac{\alpha_n-\lambda}{|\alpha_n|+1} \xi_n \right\}$. It follows that $\chi \in \D$ and $\xi=\left \{\frac{|\alpha_n|+1}{\alpha_n-\lambda} \chi_n\right \}$, hence by the boundness of the sequence $\left \{\frac{|\alpha_n|+1}{\alpha_n-\lambda}\right \}$, there exists $M>0$ such that
	$
	\n{\xi}_{\Oa}\leq M\n{\chi}_{\Oa}.
	$\\
	Moreover, denoting by $\pin{\cdot}{\cdot}_{\Oa}$ the inner product which induces the norm $\n{\cdot}_{\Oa}$,
	$$
	\pin{\chi}{\eta}_{\Oa} = \sum_{n=1}^\infty (|\alpha_n|+1)\chi_{n}\ol{\eta_n}
	= \sum_{n=1}^\infty (\alpha_n-\lambda)\xi_n\ol{\eta_n}
	= (\Oa-\lambda \iota)(\xi,\eta),
	$$
	so finally we obtain the condition {\it 2} of Theorem \ref{crit_gener_rn}, indeed
	$$\frac{1}{M}\n{\xi}_{\Oa}\leq\n{\chi}_{\Oa}  =\sup_{\n{\eta}_{\Oa}=1} \pin{\chi}{\eta}_{\Oa}= \sup_{\n{\eta}_{\Oa}=1} |(\Oa-\lambda \iota)(\xi,\eta)|.
	$$
\end{esm}

\section{Solvable sesquilinear forms with respect to an inner product}
\label{sec:sol_inn}

As we will see in next sections, solvable sesquilinear forms with respect to norms which are induced by inner products deserve particular attention. \\

\noindent Let $\O$ be a q-closed sesquilinear form with respect to a norm $\noo$ on $\D$, and suppose that $\noo$ is induced by an inner product $\pint_\O$. Since $\Eo:=\D[\pint_\O]$ is a Hilbert space and there exists $\alpha >0$ such that $\n{\xi}^2\leq \alpha \pin{\xi}{\xi}_\O$ for all $\xi\in \D$, we have that $\pint_\O$ is a closed positive sesquilinear form on $\D$. Hence, by Kato's First and Second Representation Theorems \cite[Ch. VI, Theorems 2.1, 2.23]{Kato}, there exists a positive self-adjoint operator $R$ on $\H$, with domain $D(R)\sub D(R^\mez)= \D$, and $0\in \rho(R)$, such that
\begin{align*}
\pin{\xi}{\eta}_\O &= \pin{R\xi}{\eta} \qquad \forall \xi\in D(R),\eta \in \D,\\
\pin{\xi}{\eta}_\O &= \pin{R^\mez\xi}{R^\mez\eta} \qquad \forall \xi,\eta \in \D.
\end{align*}
Moreover, $\O$ is bounded in $\Eo$, hence there exists a unique bounded operator $A\in \B(\Eo)$ such that
$$
\O(\xi,\eta)=\pin{A\xi}{\eta}_\O \qquad \forall \xi,\eta \in \D.
$$
Now, let $\Up$ be a bounded sesquilinear form in $\H$, then there exists a unique operator $B\in \B(\H)$ such that
$\Up(\xi,\eta)=\pin{B\xi}{\eta}$ for all $\xi, \eta \in \H.$\\
Therefore, taking into account that $0\in \rho(R)$,
$$
\Up(\xi,\eta)=\pin{B\xi}{\eta}=\pin{R^{-1}B\xi}{\eta}_\O \qquad \forall \xi,\eta \in \D,
$$
i.e., $R^{-1}B_{|\D}\in \B(\Eo)$ is the operator which represents the restriction $\Up_{|\D}$ of $\Up$ to $\D$, as sesquilinear form in $\Eo$. Hence
$$
(\O+\Up)(\xi,\eta)=\pin{(A+R^{-1}B_{|\D})\xi}{\eta}_\O \qquad  \forall \xi,\eta\in \D,
$$
i.e., $A+R^{-1}B_{|\D}\in \B(\Eo)$ is the operator which represents $\O+\Up$, as a sesquilinear form in $\Eo$.\\
Using conditions (\ref{cond_1_def_PO}) and (\ref{cond_2_def_PO}) one can prove the following lemma.

\begin{lem}
	\label{lem_op_Gram}
	$\Up\in \Po(\O)$ if, and only if, the operator which represents $\O+\Up$ in $\Eo$, i.e. $A+R^{-1}B_{|\D}$, is a bijection in $\D$.
\end{lem}

We suppose now that $\Up\in \Po(\O)$; then, in particular, $\O$ is solvable, and admits an associated operator $T$. By the previous lemma, we have that $A+R^{-1}B_{|\D}$ is a bijection in $\D$, and
$$
(\O+\Up)((A+R^{-1}B_{|\D})^{-1}\xi,\eta)=\pin{\xi}{\eta}_\O=\pin{R\xi}{\eta}
$$
for all $\xi\in D(R),\eta\in \D$.\\
Setting $D(T'):=(A+R^{-1}B_{|\D})^{-1}D(R)=D(RA)$ we define an operator $T'$, with domain $D(T')$, by $T':=R(A+R^{-1}B_{|\D})-B=RA$; then,
$$
(\O+\Up)(\xi,\eta)=\pin{R(A+R^{-1}B_{|\D})\xi}{\eta}=\pin{(RA+B)\xi}{\eta} \qquad \forall \xi\in D(T'),\eta \in \D,
$$
i.e.
\begin{equation}
\label{pass_form_prod_int}
\O(\xi,\eta)=\pin{RA\xi}{\eta}=\pin{T'\xi}{\eta} \qquad \forall \xi\in D(T'),\eta \in \D.
\end{equation}
We prove that $T'$ is exactly the operator $T$ associated to $\O$. From (\ref{pass_form_prod_int}) we have $T'\sub T$. Moreover, $T'+B=R(A+R^{-1}B_{|\D})$ is a bijective operator from $D(T')$ onto $\H$ and, by Theorem \ref{th_rapp_risol}, $T=T'=RA$.\\

\noindent To avoid confusion between the operators $T$ and $A$ which are both associated to $\O$, but in different spaces ($\H$ and $\Eo$), we say that $A$ is the {\it Gram operator} of $\O$ (with respect to $\pint_\O$). This terminology follows the one used in the theory of  indefinite inner product spaces (\cite[Ch. IV, par. 5]{Bognar}).

\begin{osr}
	At this point it is worth mentioning that there exist solvable sesquilinear forms with respect to norms not induced by inner products.\\
	Consider the sesquilinear form $\Oa$ of Examples $\ref{esm_form_reg_l_2(3)}$ and $\ref{esm_form_reg_l_2(4)}$. Let $p>2$, then the norm $\n{\cdot}_{\Oa}$ on $D(\Oa)$ is equivalent to the norm
	$$
	\n{\{\xi_n\}}_p=(|\xi_1|^p+|\xi_2|^p)^{\frac{1}{p}}+\left (\sum_{n=3}^\infty |\xi_n|^2 +\sum_{n=3}^\infty |\alpha_n||\xi_n|^2\right )^\mez,
	$$
	which is not induced by an inner product. By Theorem \ref{th_q_cl_sol_norm_eq}, $\Oa$ is solvable with respect to $\n{\cdot}_p$.
\end{osr}

\section{Comparisons with other representation theorems}
\label{sec:casi_part}

We have already established that closed sectorial form are solvable. But the converse is also true.

\begin{pro}
	\label{pro_solv->clos}
	Let $\O$ be a densely defined sectorial form with vertex $\gamma\in \R$, and $\lambda< \gamma$. $\O$ is closed sectorial if, and only if, it is solvable with $-\lambda \iota \in \Po(\O)$.
\end{pro}
\begin{proof}
	$"\Rightarrow"$ See Examples 5.8 and 5.12 of \cite{Tp_DB}.\\
	$"\Leftarrow"$ Let $\noo$ be a norm with respect to which $\O$ is solvable. Then, we have, for some constant $M>0$ (see \cite[Ch. VI, Sec. 1]{Kato}),
	$$
	|(\O-\lambda\iota)(\xi,\eta)|\leq M(\Re \O-\lambda \iota)(\xi,\xi)^\mez(\Re \O-\lambda \iota)(\eta,\eta)^\mez \qquad \forall \xi,\eta \in \D,
	$$
	and, from Lemma \ref{lem_crit}, there exists $c>0$ such that
	$$
	c\n{\xi}_\O \leq \sup_{\n{\eta}_\O=1} |(\O-\lambda \iota)(\xi,\eta)| \qquad \forall \xi \in \D.
	$$
	Hence, putting $\displaystyle M':=M\sup_{\n{\eta}_\O=1} (\Re\O-\lambda \iota )(\eta,\eta)^\mez$,
	$$
	c\n{\xi}_\O\leq M'(\Re\O-\lambda \iota )(\xi,\xi)^\mez \qquad \forall \xi \in \D.
	$$
	But $\Re\O$ is bounded in $\D[\noo]$, therefore for some $d>0$,
	\begin{equation}
	\label{ReO_equiv}
	c\n{\xi}_\O \leq M'(\Re\O-\lambda \iota )(\xi,\xi)^\mez\leq d\n{\xi}_\O \qquad \forall \xi\in \D.
	\end{equation}
	Consider the norm $\n{\xi}_\O':=(\Re\O-\lambda \iota )(\xi,\xi)^\mez$ on $\D$. By (\ref{ReO_equiv}), $\nor_\O'$ and $\nor_\O$ are equivalent, so the space $\D[\nor_\O']$ is a Hilbert space and $\O$ is closed.			
\end{proof}

In the rest of this section we will show that other results known in the literature  can be derived from those for solvable sesquilinear forms.\\

\noindent A sesquilinear form  $\O$ with dense domain $\D$ in $\H$ is said to be {\it closed} in the sense of McIntosh \cite{McIntosh70',McIntosh70}, if the following conditions are satisfied:
\begin{enumerate}
	\item $\D$ is a Hilbert space with an inner product $\pint_\O$;
	\item the embedding $\D[\pint_\O]\hookrightarrow\H$ is continuous;
	\item $\O$ is bounded in $\D[\pint_\O]$;
	\item there exist $\lambda \in \C$ and an bijective operator $C\in\B(\D,\pint_\O)$ such that
	$$
	(\O-\lambda \iota)(\xi,\eta)=\pin{C\xi}{\eta}_\O \qquad \forall \xi,\eta \in \D.
	$$
\end{enumerate}

According the definition of a q-closed sesquilinear form and using Lemma \ref{lem_op_Gram} one can prove the following theorem.

\begin{teor}
	\label{th_Mc_2}
	Let $\O$ be a sesquilinear form on a dense domain $\D$ in $\H$. $\O$ is closed in the sense of McIntosh if, and only if, it is solvable with respect to an inner product and $-\lambda \iota\in \Po(\O)$ for some $\lambda\in\C$.
\end{teor}

It is worth mentioning that McIntosh considers the more general case where the sesquilinear forms are defined in $X\times Y$ with $X$ and $Y$ possibly different spaces. The next example shows that there exist solvable sesquilinear forms, which are not closed in the sense of McIntosh.

\begin{esm}
	\label{esm_for_L2C_2}
	Let $\O$ be the sesquilinear form with domain
	$$
	\D:=\left \{f\in L^2(\C): \int_{\C}|x||f(x)|^2dx< \infty \right \}
	$$
	given by
	$ \displaystyle
	\O(f,g)=\int_\C xf(x)\ol{g(x)}dx$ for all $f,g \in \D$.\\
	As already seen in Example \ref{esm_for_L2C}, $\O$ is q-closed with respect to the norm
	$$ \n{f}_\O=\left (\int_{\C}(1+|x|)|f(x)|^2dx\right )^{\frac{1}{2}}.$$
	We prove that $\O$ is also solvable with respect to $\|\cdot\|_\Omega$. \\
	Let $B:L^2(\C)\to L^2(\C)$ be the bounded operator given by
	$$
	(Bf)(x)=\chi_\B(x)(1-x)f(x) \qquad x\in \C,
	$$
	where $\chi_\B$ is the characteristic function on the unit ball $\B$ centred at $0$ with radius $1$, and let
	$$
	\Up(f,g)=\pin{Bf}{g}=\int_{\C}\chi_\B(x)(1-x)f(x)\ol{g(x)}dx  \qquad \forall f,g\in L^2(\C).
	$$
	Hence, denoting by $\B^c$ the complement of $\B$ and  $r(x):=x\chi_{\B^c}(x) + \chi_\B(x)$,
	$$
	(\O+\Up)(f,g)=\int_{\C}r(x)f(x)\ol{g(x)}dx \qquad \forall f,g\in \D.
	$$
	We note that $0\notin \ol{\{r(x):x\in \C\}}$, $\frac{1}{|r(x)|}\leq 1$ and $\frac{|x|}{|r(x)|}\leq 1$ for all $x\in\C$.\\
	Let $f\in \D$ satisfy $(\O+\Up)(f,g)=0$ for all $g\in \D$. If $h\in L^2(\C)$, then $\frac{h}{\ol{r}}\in L^2(\C)$, but also $\frac{h}{\ol{r}}\in \D$, indeed
	$$
	\int_{\C}|x|\left |\frac{h}{\ol{r}}(x)\right |^2dx=\int_{\C}\frac{|x|}{|r(x)|}\frac{|h(x)|^2}{|r(x)|}dx\leq \int_{\C}|h(x)|^2dx<\infty;
	$$
	therefore, in particular,
	$$
	0=(\O+\Up)\left (f,\frac{h}{\ol{r}}\right )=\int_{\C}r(x)f(x)\ol{\frac{h}{\ol{r}}(x)}dx=\int_{\C}f(x)\ol{h(x)}dx
	$$
	for all $h\in L^2(\C)$, which implies $f=0$.\\
	Let $\Lambda$ be a bounded functional on the Hilbert space $\D[\noo]$, then by Riesz's Lemma, there exists $p\in \D$ such that
	$ \displaystyle
	\pin{\Lambda}{g}=\int_{\C} (1+|x|)p(x)\ol{g(x)}dx$ for all $g\in \D$. We set $f(x):=\frac{1+|x|}{r(x)}p(x)$. \\
	Since $\left |\frac{1+|x|}{r(x)}\right |\leq 2$ for all $x\in \C$, we have that $f\in \D$, and moreover,
	$$
	\pin{\Lambda}{g}=\int_{\C} (1+|x|)p(x)\ol{g(x)}dx=(\O+\Up)(f,g) \qquad \forall g\in \D.
	$$
	Hence, $\Up\in \Po(\O)$ and $\O$ is solvable. Now, we determine  the operator $T$ associated to $\O$. Let $\M_x$ be the multiplication operator on $L^2(\C)$ by $x$, with domain
	$$
	D(\M_x)=\left \{f\in L^2(\C):\int_{\C}|x|^2|f(x)|^2dx <\infty\right \}
	$$
	and given by $(\M_x f)(x)=xf(x)$ for all $f\in D(\M_x)$. If $f\in D(\M_x)$,  then
	$$
	2\int_{\C}|x||f(x)|^2dx\leq \int_{\C}|f(x)|^2dx+\int_{\C}|x|^2|f(x)|^2dx<\infty;
	$$
	hence, $f\in \D$ and
	$\O(f,g)=\pin{\M_x f}{g}$ for all $f\in D(\M_x), g\in \D
	$. Therefore, $\M_x\subseteq T$ by Theorem \ref{th_rapp_risol}. \\
	Moreover, from $0\notin \ol{\{r(x):x\in \C\}}$, we obtain $0\in \rho(\M_x+B)$; i.e., $\M_x+B$ has range $\H$. By Theorem \ref{th_rapp_risol} we conclude that $\M_x=T$.\\
	Note that $\O$ is solvable, but if $\lambda \in \C$ then $\Up'=-\lambda \iota\notin \Po(\O)$. Indeed, if $\Up'=-\lambda \iota\in \Po(\O)$, we would have $\lambda\in \rho(\M_x)$, but $\M_x$ has empty resolvent set.
\end{esm}

\begin{osr}
	\label{rem_unic_oper}
	In \cite[Proposition 4.2]{GKMV} it is shown that there exist different forms, that are closed in McIntosh's sense (and hence solvable), with the same associated operator. Moreover, the forms have not comparable domains. Therefore, Lemma \ref{lem_unic_op} does not hold, in general, without the hypothesis that the domains are equal. In particular, two different forms with the same associated operator are not comparable under the relation of extension.
\end{osr}

Now we prove a theorem on the relationship between solvable sesquilinear forms and Krein spaces (see \cite{Bognar}).

\begin{teor}
	\label{th_sp_Krein}
	Let $\O$ be a sesquilinear form defined on a dense domain $\D$ in $\H$, and let $\Up$ be a bounded sesquilinear form in $\H$, such that $\O+\Up$ is symmetric. Then, $\O$ is solvable with respect to an inner product and $\Up\in \Po(\O)$ if, and only if, the pair $(\D,\O+\Up)$ is a Krein space which is continuously embedded in $\H$.
\end{teor}
\begin{proof}
	By \cite[Ch. V, Theorem 1.3]{Bognar}, $(\D,\O+\Up)$ is a Krein space if, and only if, an inner product $\pint_\O$ is defined on $\D$, so that $\D[\pint_\O]$ is a Hilbert space, $\O+\Up$ is bounded in $\D[\pint_\O]$ and the Gram operator of $\O+\Up$ with respect to $\pint_\O$ is a bijection of $\D[\pint_\O]$. Hence, by the definition of a q-closed form and by Lemma \ref{lem_op_Gram}, $(\D,\O+\Up)$ is a Krein space which is continuously embedded in $\H$ if, and only if, $\O$ is a solvable with respect to $\pint_\O$ and $\Up\in \Po(\O)$.
\end{proof}

We recall that, following instead Fleige, Hassi and de Snoo \cite{FHdeS}, a symmetric sesquilinear form defined in a dense subspace $\D$ of $\H$, is said to be {\it closed}, if, for some $\lambda\in \C$, the {\em gap point}, the pair $(\D,\O-\lambda \iota)$ is a Krein space, which is continuously embedded in $\H$. From Theorem \ref{th_sp_Krein} we have the following comparison between closed forms and solvable forms.

\begin{cor}
	Let $\O$ be a symmetric sesquilinear form defined on a dense domain $\D$  in $\H$, and let $\lambda\in \R$. $\O$ is closed, with gap point $\lambda$, if, and only if, $\O$ is solvable with respect to an inner product and $-\lambda \iota\in \Po(\O)$.
\end{cor}

Consequently, Lemma 3.1 and Theorem 3.3 of \cite{FHdeS} are special cases of Theorems \ref{th_equiv_norm_q-chius},  \ref{th_rapp_risol} and of Corollaries \ref{cor_th_rapp_est}, \ref{cor_auto}. Proposition 3.2 in \cite{FHdeSW} corresponds to Theorem \ref{lem_unic_op}. The next example shows, instead, a symmetric solvable form, which is not closed in the sense of \cite{FHdeS}.

\begin{esm}
	Let $\H=L^2(\R)$. We define the symmetric sesquilinear form $\O$ with domain
	$$
	\D:=\left \{f\in L^2(\R): \int_{\R}|x||f(x)|^2dx< \infty \right \}
	$$
	putting
	$\displaystyle
	\O(f,g)=\int_\R xf(x)\ol{g(x)}dx
	$ for all $f,g\in \D$.\\
	$\O$ is solvable with respect to the norm (induced by an inner product) $\n{f}_\O=(\int_{\R}(1+|x|)|f(x)|^2dx)^\frac{1}{2}$, $f\in \D$. Indeed, in a way similar to Example \ref{esm_for_L2C}, we prove that the bounded sesquilinear form $\Up=-i\iota$ is in $\Po(\O)$. Moreover, the operator associated to $\O$ is the multiplication operator $\M_x$ by $x$ on $L^2(\R)$, with domain
	$$
	D(\M_x)=\left \{f\in L^2(\R): \int_{\R}x^2|f(x)|^2dx< \infty\right \}
	$$
	and given by
	$(\M_x f)(x)=xf(x)$ for all $f\in D(\M_x)$.\\
	However, $\O$ is not closed in the sense of \cite{FHdeS}, because if it were, with some gap point $\lambda\in \R$, then $\lambda$ would belong to the resolvent set of $\M_x$, which does not contain real numbers.
\end{esm}

We recall that an operator $T$ is  {\it accretive} if its numerical range is contained in the half-plane $\{\lambda\in \C:\Re \lambda \geq 0\}$, and  {\it maximal accretive} if it has no proper accretive extension. A characterization of maximal dissipative operators (i.e., the opposite of those maximal accretive) can be found in \cite{Phillips}.\\
The following theorem is the result of McIntosh in \cite{McIntosh68} on the representation of certain accretive sesquilinear forms (as for operators, accretive means that the numerical range is contained in the half-plane $\{\lambda\in \C:\Re \lambda \geq 0\}$). We prove that this is another special case of Theorem \ref{th_rapp_risol}.

\begin{teor}
	Let $\O$ be an accretive sesquilinear form on a dense domain $\D$  in $\H$. Suppose that $\D$ can be made into a Hilbert space (indicated by $\H_1$), with norm $\nor_1$, satisfying the following conditions
	\begin{enumerate}
		\item the embedding $\H_1\hookrightarrow \H$ is continuous;
		\item $\O$ is bounded in $\H_1$;
		\item if $\{\xi_n\}$ is a sequence in $\D$ such that, $\displaystyle \sup_{\n{\eta}_1=1} |(\O+\iota)(\xi_n,\eta)|\to 0$ then $\n{\xi_n}_1\to 0$.
	\end{enumerate}
	Then, there exists a maximal accretive operator $T$, with dense domain $D(T)\sub \D$ in $\H$, such that
	$\O(\xi,\eta)=\pin{T\xi}{\eta}$  for all $\xi\in D(T),\eta \in \D$.
\end{teor}
\begin{proof}
	Hypotheses {\it 1, 2} imply that $\O$ is a q-closed sequilinear form with respect to $\n{\cdot}_1$. One can prove that $\O$ is solvable with respect to $\n{\cdot}_1$, using Corollary \ref{cor_crit_qc'}. Indeed, $-1\notin \rn_\O$, the numerical range of $\O$, and the third hypothesis in the statement is exactly the condition {\it 1} in Theorem \ref{crit_gener_rn}. So, the required properties follow from Theorem \ref{th_rapp_risol}, and $T$ is a maximal accretive operator by a corollary of \cite[Theorem 1.1.1]{Phillips}.
\end{proof}

Now consider the sesquilinear forms of the type
\begin{equation}
\label{form_QH1H2}
\O(\xi,\eta)=\pin{HA^\mez\xi}{A^\mez\eta} \qquad \forall \xi,\eta \in \D:=D(A^\mez),
\end{equation}
studied in \cite{GKMV}, where $H,A$ are self-adjoint, $H$ is bounded, $A$ is positive and $0\in \rho(H)\cap \rho(A)$. It is easy to see, using the definition, that $\O$ are solvable with respect to the norm  $\n{\xi}_\O=\n{A^\mez \xi}$ for all $\xi \in \D$, and $\Up\in \Po(\O)$, where $\Up=0$.\\

\noindent In  \cite{Schmitz} the sesquilinear forms considered are defined by
\begin{equation*}
\O(\xi,\eta)=\pin{HA^\mez\xi}{A^\mez\eta} \qquad \forall \xi,\eta \in \D:=D(A^\mez),
\end{equation*}
where the unique difference from above is that $A$ is only non-negative and self-adjoint, but the existence of a sef-adjoint involution $J$ with some property is required. As one can see in the proof of Theorem 2.3 of \cite{Schmitz}, such sesquilinear forms are perturbations of those defined by (\ref{form_QH1H2}) with the bounded form $\Up(\xi,\eta)=\pin{J\xi}{\eta}$ for all $\xi,\eta \in \H$. Hence, $\O$ are solvable and $\Up\in \Po(\O)$.\\	 

\noindent We recall that the spectrum of a closed operator is called {\it purely discrete} if it contains only eigenvalues of finite multiplicity which have no finite accumulation points (see \cite[Ch. 2]{Schm}). We prove now the following theorem.

\begin{teor}
	\label{th_compat}
	Let $\O$ be a solvable sesquilinear form on $\D$ with respect to an inner product $\pint_\O$ and let $T$ be its associated operator. Suppose that the embedding $\D[\pint_\O]\hookrightarrow \H$ is compact. \\
	Then, for all $A\in \B(\H)$, such that $T-A$ is invertible with bounded inverse $(T-A)^{-1}\in \B(\H)$, one has that $(T-A)^{-1}$ is compact. \\
	In particular, if $\lambda \in \rho(T)$, then $R_\lambda (T)=(T-\lambda I)^{-1}$ is compact. \\
	If the spectrum $\sigma(T)$  of $T$ is different from $\C$, then $\sigma(T)$ is purely discrete.	
\end{teor}
\begin{proof}
	Let $\noo$ be the norm induced by $\pint_\O$, and $\Eo=\D[\noo]$. As we have seen in Section \ref{sec:sol_inn}, there exists a positive self-adjoint operator $R$, with bounded inverse, such that $\D=D(R^\frac{1}{2})$ and $\n{\xi}_\O=\n{R^\frac{1}{2}\xi}$ for all $ \xi \in \D$.\\
	Hence, since $R^{-\frac{1}{2}}\in \B(\H)$,
	\begin{equation}
	\label{pas_comp}
	\n{R^{-\frac{1}{2}}\eta}_\O=\n{\eta} \qquad \forall \eta \in \H.
	\end{equation}
	The compactness of the embedding $\Eo\hookrightarrow \H$ and (\ref{pas_comp}) imply that $R^{-\frac{1}{2}}$ is a compact operator. Indeed, if $\{\xi_n\}\subset D(R^{-\frac{1}{2}})=\H$ is a bounded sequence in $\H$, then since $\n{R^{-\frac{1}{2}}\xi_n}_\O=\n{\xi_n},$
	the sequence  $\{R^{-\frac{1}{2}}\xi_n\}$ is bounded in $\Eo$, therefore $\{R^{-\frac{1}{2}}\xi_n\}$ admits a subsequence which converges in $\H$.\\
	$T$ is the associated operator to $\O$, hence (see again Section \ref{sec:sol_inn}) there exist $B\in \B(\H)$ and an isomorphism $C\in \B(\Eo)$ such that $T+B=RC_{|D(T)}.$\\
	Thus
	\begin{equation}
	\label{pas_comp2}
	(T+B)^{-1}=C_{|D(T)}^{-1}R^{-1}=C_{|D(T)}^{-1}R^{-\frac{1}{2}}R^{-\frac{1}{2}}.
	\end{equation}
	The operator $C_{|D(T)}^{-1}R^{-\frac{1}{2}}$ is bounded in $\H$, because
	$$
	\n{C_{|D(T)}^{-1}R^{-\frac{1}{2}} \xi}\leq \alpha \n{C_{|D(T)}^{-1}R^{-\frac{1}{2}} \xi}_\O\leq \alpha \delta \n{R^{-\frac{1}{2}} \xi}_\O=\alpha \delta \n{\xi} \qquad \forall \xi \in \H,
	$$
	(we have used the facts that $\n{\cdot}\leq\alpha \noo$ and that $\n{{C^{-1}\xi}}_\O\leq \delta \n{\xi}_\O$, with certain constants $\alpha,\delta$). Hence from (\ref{pas_comp2}) and from the compactness of $R^{-\mez}$ we see that $(T+B)^{-1}$ is compact.\\
	Let $A\in \B(\H)$ be such that $T-A$ is invertible with bounded inverse $(T-A)^{-1}\in \B(\H)$. Put, for simplicity of notation, $R_A(T):=(T-A)^{-1}$, hence $R_{-B}(T)=(T+B)^{-1}$. We have the following relation (which is analogous to the classical one between resolvent operators)
	$$
	R_A(T)(A+B)R_{-B}(T)= R_A(T)(T+B-(T-A))R_{-B}(T)=R_A(T)-R_{-B}(T),
	$$
	which shows that $R_A(T)$ is compact.\\
	Suppose now that the spectrum $\sigma(T)$ is different from $\C$, and let $\lambda \in \rho(T)$. Then $R_\lambda (T)=(T-\lambda I)^{-1}$ is compact, and from \cite[Proposition 2.11]{Schm} it follows that $\sigma(T)$ is purely discrete.
\end{proof}

\begin{cor}
	In the hypothesis of Theorem \ref{th_compat} and assuming $\H$ is infinite dimensional, if $\lambda \in \rho(T)$ and the resolvent $R_\lambda(T)$ is normal, then there exist a sequence $\{\mu_n: n \in \N\}$ of eigenvalues of $T$, such that $\displaystyle \lim_{n\to\infty}|\mu_n| = +\infty$, and an orthonormal basis $\{\xi_n : n \in \N\}$ of
	$\H$ which consists of eigenvectors of $T$.
\end{cor}
\begin{proof}
	This fact is a consequence of \cite[Theorem A.4]{Schm}.
\end{proof}

\begin{osr}
	Let $\O$ be a solvable sesquilinear form on $\D$ with respect to a norm $\noo$, and with associated operator $T$. Assume that $\lambda$ and $\xi$ are an eigenvalue and a corresponding eigenvector of $T$, respectively. Then, they are an eigenvalue and a corresponding eigenvector of $\O$, respectively; in other words,	$\O(\xi,\eta)=\lambda \pin{\xi}{\eta}$ for all $\eta \in \D$.\\
	The converse is also true. Indeed, if $\lambda$ and $\xi$ are an eigenvalue and a corresponding eigenvector of $\O$, respectively, then, by Corollary \ref{cor_th_rapp_est}, we have $\xi\in D(T)$  and $T\xi=\lambda\xi$.\\
	Hence, the statement of the previous corollary can be formulated considering eigenvalues and eigenvectors of the form.
\end{osr}

\begin{defin}
	A q-closed sesquilinear form $\O$ on $\D$ with respect to a norm $\noo$ is called {\it coercive} on $\D[\noo]$ if there exists $\gamma >0$ such that $|\O(\xi,\xi)|\geq\gamma \n{\xi}_\O^2$ for all $\xi\in \D$.
\end{defin}

The following facts on coercive sesquilinear forms hold.

\begin{teor}
	\label{th_coer}
	Let $\O$ be a q-closed sesquilinear form on $\D$ with respect to a norm $\noo$. If there exists a bounded form $\Up$ on $\H$ such that $\O+\Up$ is coercive on $\D[\noo]$, then $\O$ is solvable, and in particular, $\Up\in \Po(\O)$.
\end{teor}
\begin{proof}
	Since there exists $\gamma>0$ such that $|(\O+\Up)(\eta,\eta)|\geq\gamma \n{\eta}_\O^2$ for all $\eta \in \D$, we have that $N(\O+\Up)=\{0\}$. Moreover, for all $\eta \in \D$
	\begin{equation}
	\label{pass_coerc}
	\gamma\n{\eta}_\O\leq\sup_{\n{\xi}_\O=1} |(\O+\Up)(\xi,\eta)|.
	\end{equation}
	Indeed, if $\eta=0$ then (\ref{pass_coerc}) is obvious. If $\eta\neq 0$ then
	$$
	\gamma\n{\eta}_\O= \frac{\gamma\n{\eta}_\O^2}{\n{\eta}_\O}\leq \frac{ |(\O+\Up)(\eta,\eta)|}{\n{\eta}_\O}\leq \sup_{\n{\xi}_\O=1} |(\O+\Up)(\xi,\eta)|.
	$$
	Hence, by Theorem  \ref{lem_crit}, $\Up\in \Po(\O)$ and $\O$ is solvable.
\end{proof}

\begin{cor}
	\label{cor_coerc}
	A q-closed sesquilinear form $\O$ on $\D$ with respect to a norm $\noo$ and coercive on $\D[\noo]$ is solvable, and $0\in \rho(T)$, the resolvent set of the operator $T$ associated to $\O$.
\end{cor}

In general there is no relations between the concept of $j$-{\it elliptic} form in \cite{Arendt} and that of {\it solvable} form, because a $j$-elliptic form is defined on a Hilbert space $(V,\nor_V)$, where $V$ need not be a subspace of $\H$. \\
Let us consider the case when $V$ is a dense subspace of $\H$ and $j:V\to \H$ is the embedding. Let $\O$ be a $j$-{\it elliptic} sesquilinear form on $V$ ($\O$ is q-closed with respect to $\nor_V$), then there exists $\omega \in \R, \mu> 0$ such that
$\Re \O(\xi,\xi)+\omega \n{\xi}^2\geq \mu \n{\xi}_V^2$ for all $\xi \in V$.
Therefore,
$$
|\O(\xi,\xi)+\omega \n{\xi}^2|\geq \mu \n{\xi}_V^2 \qquad \forall \xi \in V,
$$
and, by Theorem \ref{th_coer}, $\O$ is solvable and $\omega \iota \in \Po(\O)$.\\

\noindent We return again to the problem of forms with the same associated operators (see Lemma \ref{lem_unic_op} and Remark \ref{rem_unic_oper}).

\begin{pro}
	\label{pro_form_coer_1_oper}
	Let $\O_1$ be a solvable sesquilinear form on $\D_1$ with respect to a norm $\nor_1$ and let $\O_2$ be a solvable sesquilinear form on $\D_2$ with respect to $\nor_2$. If $\O_1$ and $\O_2$ are coercive in $\D_1[\nor_1]$ and in $\D_2[\nor_2]$, respectively, and they have the same associated operator, then $\D_1=\D_2$ and $\O_1=\O_2$.
\end{pro}
\begin{proof}
	There exist constants $\beta_1,\beta_2,\gamma_1,\gamma_2>0$ such that
	$$
	\gamma_1\n{\xi}_1^2\leq |\O_1(\xi,\xi)|\leq \beta_1 \n{\xi}_1^2 \qquad \forall \xi \in \D_1
	$$
	$$
	\gamma_2\n{\eta}_2^2\leq |\O_2(\eta,\eta)|\leq \beta_2 \n{\eta}_2^2 \qquad \forall \eta \in \D_2.
	$$
	Let $D(T)$ be the domain of the common associated operator $T$; then we have
	$$
	\gamma_2\n{\xi}_2^2\leq |\pin{T\xi}{\xi}|\leq \beta_1 \n{\xi}_1^2 \qquad \forall \xi \in D(T),
	$$
	and a similar expression exchanging $\nor_1$ and $\nor_2$. Therefore, the two norms are equivalent in $D(T)$, but $D(T)$ is dense both in $\D_1[\nor_1]$ and in $\D_2[\nor_2]$. Hence, $\D_1=\D_2$ and $\O_1=\O_2$ by Lemma \ref{lem_unic_op}.
\end{proof}

By Corollary \ref{cor_coerc}, hence Theorem 11.3 of \cite{Schm} is a consequence of Theorems \ref{th_rapp_risol}, \ref{th_ris_agg} and \ref{th_compat}. Moreover, Corollaries 11.4 and 11.5 of \cite{Schm} are specific cases of Proposition \ref{pro_form_coer_1_oper} and Corollary \ref{cor_auto<->simm}, respectively.

\bigskip
{\sc Acknowledgement:} This work has been done in the framework of the project ''Problemi spettrali e di rappresentazione in quasi *-algebre di operatori'', INDAM-GNAMPA 2017.

\vspace*{0.5cm}
\textsc{Rosario Corso, Dipartimento di Matematica e Informatica, Università degli Studi di Palermo, I-90123 Palermo, Italy}

{\it E-mail address}: {\bf rosario.corso@studium.unict.it} \\

\textsc{Camillo Trapani, Dipartimento di Matematica e Informatica, Università degli Studi di Palermo, I-90123 Palermo, Italy}

{\it E-mail address}: {\bf camillo.trapani@unipa.it}


\begin{thebibliography}{100}
	\addcontentsline{toc}{chapter}{References}

	
	\bibitem{AITp}
	J-P. Antoine, A. Inoue, C. Trapani, {\it Partial *-algebras and their operator realizations}, Kluwer, Dordrecht, 2002.
	
	\bibitem{ATp}
	J-P. Antoine, C. Trapani, {\it Partial Inner Product Spaces - Theory and Applications}, Lecture Notes in Math., vol. 1986, Springer, Berlin, Heidelberg, 2009.
	
	\bibitem{Arendt}
	W. Arendt, A.F.M. ter Elst, {\it Sectorial forms and degenerate differential operators}, J. Operator Theory, 67 (2012), 33–72.
		
	\bibitem{Tp_DB}
	S. Di Bella, C. Trapani, {\it Some representation theorems for sesquilinear forms}, J. Math. Anal. Appl., 451 (2017), 64-83.
		
	\bibitem{Bognar}
	J. Bogn\'{a}r, {\it Indefinite inner product spaces}, Ergeb. Math. Grenzgeb., 78, Springer-Verlag, New York–Heidelberg, 1974.
	
	\bibitem{Fleige}
	A. Fleige, {\it Non-semibounded sesquilinear forms and left-indefinite Sturm-Liouville	problems}, Integral Equations Operator Theory, 33 (1999), 20–33.
	
	\bibitem{FHdeS}
	A. Fleige, S. Hassi, H. de Snoo, {\it A Krein space approach to representation theorems and generalized Friedrichs extensions}, Acta Sci. Math. (Szeged), 66 (2000), 633-650.
	
	\bibitem{FHdeSW'}
	A. Fleige, S. Hassi, H.S.V. de Snoo, H. Winkler, {\it Sesquilinear forms corresponding to a
	non-semibounded Sturm-Liouville operator}, Proc. Roy. Soc. Edinburgh Sect. A, 140A (2010), 291-318.
		
	\bibitem{FHdeSW}
	A. Fleige, S. Hassi, H. de Snoo, H. Winkler, {\it Non-semibounded closed symmetric forms associated with a generalized Friedrichs extension}, Proc. Roy. Soc.  Edinburgh Sect. A, Mathematics, 144(4), (2014), 731–745.
	
	\bibitem{GKMV}
	L. Grubi\u{s}i\'{c}, V. Kostrykin, K. A. Makarov, K. Veseli\'{c}, {\it Representation theorems for
	indefinite quadratic forms revisited}, Mathematika, 59(1), (2013), 169–189.
	
	\bibitem{Kato}
	T. Kato, {\it Perturbation Theory for Linear Operators}, Springer, Berlin, 1966.
	
	\bibitem{McIntosh68}
	A. McIntosh, {\it Representation of bilinear forms in Hilbert space by linear operators}, Trans. Amer. Math. Soc., 131, 2 (1968), 365-377.
	
	\bibitem{McIntosh70'}
	A. McIntosh, {\it Bilinear forms in Hilbert space}, J. Math. Mech., 19 (1970), 1027–1045.
	
	\bibitem{McIntosh70}
	A. McIntosh, {\it Hermitian bilinear forms which are not semibounded}, Bull.	Amer. Math. Soc., 76 (1970), 732–737.
	
	\bibitem{Megginson}
	R. E. Megginson, {\it An Introduction to Banach Space Theory}, Springer-Verlang, New York-Heidelberg-Berlin, 1998.

	\bibitem{Phillips}
	R. S. Phillips, {\it Dissipative operators and hyperbolic systems of partial differential equations}, Trans. Amer. Math. Soc. 90 (1959), 193-254.
	
	\bibitem{Schmitz}
	S. Schmitz, {\it Representation Theorems for Indefinite Quadratic Forms Without Spectral Gap}, Integral Equations Operator Theory, 83 (2015), 73–94.
	
	\bibitem{Schm}
	K. Schm\"{u}dgen, {\it Unbounded Self-adjoint Operators on Hilbert Space}, Springer, Dordrecht, 2012.
	
	\bibitem{Weid}
	J. Weidmann, {\it Linear Operators in Hilbert Spaces}, Springer-Verlag,  1980.
	
\end{thebibliography}
\end{document}